\documentclass[smallextended]{svjour3}       
\smartqed  
\usepackage{hyperref}
\usepackage{lipsum}
\usepackage{latexsym,amsfonts,amsmath,amssymb}
\usepackage{graphicx}
\usepackage{color,caption,url}
\usepackage[caption=false]{subfig}
\usepackage{epstopdf}
\usepackage{booktabs,ctable,threeparttable,boxedminipage}
\usepackage{algorithm}
\usepackage{algpseudocode}
\usepackage{qbordermatrix}
\usepackage{blkarray}
\usepackage{enumerate}
\usepackage{arydshln}
\usepackage{lineno}
\usepackage{float}  

\newtheorem{remark1}{Remark}[section]

\numberwithin{figure}{section}
\numberwithin{table}{section}
\numberwithin{equation}{section}
\numberwithin{theorem}{section}
\numberwithin{lemma}{section}
\numberwithin{algorithm}{section}
\numberwithin{remark1}{section}

\newcommand\argmin{\mathop{\text{argmin}}}
\newcommand\bA{\boldsymbol{A}}
\newcommand\bb{\boldsymbol{b}}
\newcommand\bx{\boldsymbol{x}}
\newcommand\bI{\boldsymbol{I}}
\newcommand\bR{\boldsymbol{R}}
\newcommand\bQ{\boldsymbol{Q}}
\newcommand\bC{\boldsymbol{C}}
\newcommand\bL{\boldsymbol{L}}
\newcommand\bU{\boldsymbol{U}}
\newcommand\bZ{\boldsymbol{Z}}
\newcommand\bV{\boldsymbol{V}}
\newcommand\bW{\boldsymbol{W}}
\newcommand\bM{\boldsymbol{M}}
\newcommand\bN{\boldsymbol{N}}
\newcommand\bB{\boldsymbol{B}}
\newcommand\bT{\boldsymbol{T}}
\newcommand\bD{\boldsymbol{D}}

\newcommand\bF{\boldsymbol{F}}
\newcommand\bG{\boldsymbol{G}}
\newcommand\bH{\boldsymbol{H}}
\newcommand\bP{\boldsymbol{P}}
\newcommand\bp{\boldsymbol{p}}
\newcommand\bJ{\boldsymbol{J}}
\newcommand\bGamma{\boldsymbol{\Gamma}}
\newcommand\bff{\boldsymbol{f}}
\newcommand\br{\boldsymbol{r}}
\newcommand\bs{\boldsymbol{s}}
\newcommand\bz{\boldsymbol{z}}
\newcommand\bu{\boldsymbol{u}}
\newcommand\bv{\boldsymbol{v}}
\newcommand\bw{\boldsymbol{w}}
\newcommand\by{\boldsymbol{y}}
\newcommand\be{\boldsymbol{e}}
\newcommand\bd{\boldsymbol{d}}
\newcommand\bq{\boldsymbol{q}}
\newcommand\bg{\boldsymbol{g}}
\newcommand\bSigma{\boldsymbol{\Sigma}}
\newcommand\bDelta{\boldsymbol{\Delta}}
\newcommand\bepsilon{\boldsymbol{\epsilon}}
\newcommand\bLambda{\boldsymbol{\Lambda}}

\ifpdf
\DeclareGraphicsExtensions{.eps,.pdf,.png,.jpg}
\else
\DeclareGraphicsExtensions{.eps}
\fi
\let\oldequation\equation
\let\oldendequation\endequation
\renewenvironment{equation}
{\linenomathNonumbers\oldequation}
{\oldendequation\endlinenomath}
\let\oldalign\align
\let\oldendalign\endalign
\renewenvironment{align}
{\linenomathNonumbers\oldalign}
{\oldendalign\endlinenomath}
\begin{document}
	
	
	\title{Subspace projection regularization for large-scale Bayesian linear inverse problems
}
	
	\titlerunning{Subspace projection regularization}  
	
	\author{Haibo Li}
	
	
	\institute{Haibo Li \at
		School of Mathematics and Statistics, The University of Melbourne, Parkville, VIC 3010, Australia. \\
		\email{haibo.li@unimelb.edu.au}   
	}
	
	
	\maketitle

\begin{abstract}
	The Bayesian statistical framework provides a systematic approach to enhance the regularization model by incorporating prior information about the desired solution. For the Bayesian linear inverse problems with Gaussian noise and Gaussian prior, we propose an iterative regularization algorithm within the subspace projection regularization (SPR) framework. By treating the forward model matrix as a linear operator between two proper finite dimensional Hilbert spaces, we first introduce an iterative process that can generate a series of valid solution subspaces. The SPR method then projects the original problem onto these solution subspaces to get a series of low dimensional linear least squares problems, where an efficient procedure is developed to update the solutions of them to approximate the desired solution of the original problem. With the newly designed early stopping rules, this iterative algorithm can obtain a regularized solution with satisfied accuracy. We establish several theoretical results to reveal the regularization properties of the algorithm. Both small-scale and large-scale inverse problems are used to test the proposed algorithm and demonstrate its robustness and efficiency. The most computationally intensive operations in the proposed algorithm only involve matrix-vector products, making it highly efficient for large-scale problems. 
\keywords{Bayesian inverse problem \and Gaussian prior \and subspace projection regularization \and generalized Golub-Kahan bidiagonalization \and semi-convergence \and early stopping}
\subclass{62F15 \and 65J22 \and 65F22 \and 65F10}
\end{abstract}

\section{Introduction}
Inverse problems arise frequently in various scientific and engineering fields such as medical imaging, computed tomography, geoscience, data assimilation and so on \cite{Kaip2006,Hansen2006,Buzug2008,Law2015,Richter2016}. In these problems, we aim to reconstruct unknown parameters or functions from indirect measurements that often include noise. In this paper, we consider linear inverse problems of the discrete form
\begin{equation}\label{inverse1}
	\bb = \bA\bx+\bepsilon, \ \ \bepsilon\sim\mathcal{N}(\boldsymbol{0}, \bM),
\end{equation}
where $\bx\in\mathbb{R}^{n}$ is the underlying quantity we wish to reconstruct, $\bA\in\mathbb{R}^{m\times n}$ is the discretized forward model matrix, $\bb\in\mathbb{R}^{m}$ is the vector of observation,  and $\bepsilon$ is the noise that follows a zero mean Gaussian distribution with positive definite covariance matrix $\bM$. In the sense of Hadamard \cite{hadamard1923lectures}, this problem is usually ill-posed, which means that there may be multiple solutions that fit the observation equally well, or the solution is very sensitive respect to observation perturbation. The ill-posedness of inverse problems is a big challenge that prohibits us from computing an accurate solution by a direct inversion method.

Regularization is a very common technique to overcome ill-posedness, where the core idea is to use prior knowledge about the underlying solution to constrain the set of possible solutions and improve their stability and uniqueness properties. Here we deal with the inverse problem \eqref{inverse1} and its regularization from a Bayesian perspective \cite{Kaip2006,Stuart2010}, where the unknown and observation are modeled as random variables, and prior knowledge and assumptions about the unknown are combined with the observation to obtain a posterior probability distribution of $\bx$. By \eqref{inverse1}, the observation vector $\bb$ has a conditional probability density function (pdf)
\begin{equation*}
	p(\bb|\bx) \propto \exp\left(-\frac{1}{2}\|\bA\bx-\bb\|_{\bM^{-1}}^{2}\right).
\end{equation*}
Maximizing $p(\bb|\bx)$ with respect to $\bx$ is known as the maximum likelihood estimate (MLE), where $\|\bx\|_{\bB}:=(\bx^{\top}\bB\bx)^{1/2}$ is the $\bB$-norm of $\bx$ for a positive definite matrix $\bB$. To regularize MLE, we assume a Gaussian prior about the desired solution with the form $\bx\sim\mathcal{N}(\boldsymbol{0}, \lambda^{-1}\bN)$, where $\bN$ is a positive definite covariance matrix. Using the Bayes' formula that estimates posterior pdf of $\bx$ after observation $\bb$, we have 
\begin{equation*}
	p(\bx|\bb,\lambda) 
	\propto p(\bx|\lambda)p(\bb|\bx)
	\propto \exp\left(-\frac{1}{2}\|\bA\bx-\bb\|_{\bM^{-1}}^{2}-\frac{\lambda}{2}\|\bx\|_{\bN^{-1}}^2\right).
\end{equation*}
The maximum a posterior estimate (MAP) of $\bx$ finds a regularized solution to \eqref{inverse1} by maximizing the posterior pdf of $\bx$, which leads to
\begin{equation}\label{Bayes1}
	\min_{\bx \in \mathbb{R}^{n}}\{\|\bA\bx-\bb\|_{\bM^{-1}}^{2} + \lambda\|\bx\|_{\bN^{-1}}^2\}.
\end{equation}
The form of \eqref{Bayes1} is the same as the Tikhonov regularization \cite{Tikhonov1977}, where the regularization term $\|\bx\|_{\bN^{-1}}^2$ enforces extra structure on the solution that comes from the prior distribution of $\bx$. The regularization parameter $\lambda$ controls the trade-off between the data-fit term and regularization term, which should be determined in advance.

The Bayesian statistical framework offers a systematic approach to enhance the regularization model by incorporating prior information about the solution. However, developing efficient algorithms for computing those informative solutions is very challenging. Firstly, the quality of the reconstructed solution heavily relies on the accuracy of the regularization parameter. Although various methods have been proposed, such as the L-curve (LC) method \cite{Hansen1992}, generalized cross-validation (GCV) \cite{Golub1979} and discrepancy principle (DP) \cite{Morozov1966}, but each has its limitations and assumptions, and these methods require solving \eqref{Bayes1} for many different values of $\lambda$. Secondly, to solve \eqref{Bayes1}, the Cholesky factorizations $\bM^{-1}=\bL_{M}^{\top}\bL_{M}$ and $\bN^{-1}=\bL_{N}^{\top}\bL_{N}$ are required to transform \eqref{Bayes1} to the standard 2-norm form
\begin{equation}\label{gen_regu}
	\min_{\bx \in \mathbb{R}^{n}}\{\|\bL_{M}(\bA\bx-\bb)\|_{2}^{2} + \lambda\|\bL_{N}\bx\|_{2}^2\}
\end{equation}
if an iterative solver for the above least squares problem is exploited; or $\bM^{-1}$ and $\bN^{-1}$ must be computed if an iterative solver for the normal equation of \eqref{Bayes1} is exploited. For both the cases, it is not feasible to work with the matrix inversion and factorization if its dimension is very large or the explicit form of the matrix is not available.

Up to now, there are many numerical methods that deal with large-scale Bayesian linear inverse problems; here we review iterative methods that are based upon Krylov subspaces \cite{liesen2013krylov}. The main idea behind these methods is to project the original linear system \eqref{inverse1} onto a series of suitable low dimensional subspaces and solve these small projected problems step by step, and the iteration should be terminated early to avoid too many noisy components included in the solution \cite{Kaip2006,Calvetti2018}. In this approach, a key point is that these subspaces must be elaborately constructed by an iterative process that can incorporate information about the prior distribution of $\bx$. To this end, the priorconditioning technique is proposed, where the statistically inspired left and right priorconditioners are used to whiten the noise and unknown, respectively \cite{calvetti2005priorconditioners,calvetti2017priorconditioned,Calvetti2018}. However, it usually requires that $\bM^{-1}$ and its Cholesky factorization can be computed easily, and $\bL_{N}$ instead of $\bN^{-1}$ should be explicitly constructed as the right priorconditioner. In some scenarios of applications, the covariance matrix $\bN$ instead of $\bL_{N}$ is used to add prior information about $\bx$, such as the Met\'{e}rn class of covariance matrices \cite{Genton2001classes,Roininen2011correlation,antil2023efficient}. This raises the need for new methods to construct suitable solution subspaces to incorporate prior information about $\bx$ while avoiding Cholesky factorizations and inversions of the prior covariance matrices.

Recently, a generalized hybrid iterative method is proposed for large-scale Bayesian linear inverse problems \cite{Chung2017}, which is based on an iterative process called the generalized Golub-Kahan bidiagonalization (gen-GKB) \cite{Arioli2013}. This method first transforms \eqref{Bayes1} to
\begin{equation}\label{trans_tikh}
	\min_{\bar{\bx} \in \mathbb{R}^{n}}\{\|\bA\bN\bar{\bx}-\bb\|_{\bM^{-1}}^{2} + \lambda\|\bar{\bx}\|_{\bN}^2\}
\end{equation}
by the substitution $\bx=\bN\bar{\bx}$ to avoid the inversion of $\bN$. Then it uses gen-GKB to generate a series of $\bN$-orthonormal bases of $k$-dimensional Krylov subspaces for $k=1,2,\dots$ and projects \eqref{trans_tikh} onto these subspaces to get a series of $k$-dimensional Tikhonov regularization problems. To avoid choosing a good regularization parameter in advance, this method adapts a hybrid regularization approach, which determines a regularization parameter $\lambda_k$ for the projected $k$-dimensional problem at each $k$-th step using a parameter choosing criterion such as GCV \cite{Kilmer2001}. By determining $\lambda_k$ at each iteration, this method computes a series of iterative solutions which are used to approximate the best solution to \eqref{trans_tikh} with the optimal parameter $\lambda_{opt}$.


The gen-GKB based hybrid regularization method is reasonable in the sense that the iterative solution lies within the same subspace as the MAP estimator of the original Bayesian inverse problem \eqref{Bayes1}, thereby the prior information of $\bx$ encoded by the Gaussian prior can be well incorporated into the iterative solution. Furthermore, the computations of $\bN^{-1}$ and factorizations of $\bM^{-1}$ and $\bN^{-1}$ are not needed, which makes it very practical for solving large-scale inverse problems. Although this method is effective for many types of inverse problems, sometimes the regularization parameter $\lambda_k$ can not be well determined, resulting in that $\lambda_k$ does not converge to $\lambda_{opt}$ or converges very slowly, and the corresponding iterative solution have a poor accuracy; this is a common potential flaw for hybrid methods \cite{Chungnagy2008,Renaut2017}. Besides, this method needs to transform \eqref{Bayes1} to \eqref{trans_tikh} first to apply gen-GKB, and there is not a direct motivation provided to explain why gen-GKB should be used to generate valid solution subspaces for \eqref{Bayes1}.

In order to avoid choosing regularization parameter in advance for \eqref{Bayes1} or step by step for the generalized hybrid method, we consider the iterative subspace projection regularization (SPR) method \cite[\S 3.3]{Engl2000}. For the Bayesian linear inverse problem \eqref{Bayes1}, the core idea is to seek a series of $\bx_k$ as the solution to
\begin{equation*}
	\min_{\bx\in\mathcal{X}_k}\|\bx\|_{\bN^{-1}}, \ \ \mathcal{X}_k = \{\bx: \min_{\bx\in\mathcal{S}_k}\|\bA\bx-\bb\|_{\bM^{-1}}\} ,
\end{equation*}
where $\mathcal{S}_k$ is the subspace of $\mathbb{R}^{n}$ of dimension $k=1,2,\dots$, and the iteration proceeds until an early stopping rule is satisfied to overcome too many noisy components are contained. We call $\mathcal{S}_k$ the \textit{solution subspace}, which should be constructed elaborately such that the information about the Gaussian noise $\bepsilon$ and desired property of $\bx$ encoded by the Gaussian prior can be incorporated into it. In this paper, we propose an iterative regularization algorithm within the SPR framework based on an iterative process that is equivalent to gen-GKB. This new algorithm exhibits good performance for indirectly solving the Bayesian linear inverse problem \eqref{Bayes1} and is very efficient for large-scale problems.

The main contributions of this paper are listed as follows:
\begin{itemize}
	\item[$\bullet $] By treating the matrix $\bA$ as a bounded linear operator between the two Hilbert spaces $(\mathbb{R}^{n}, \langle\cdot,\cdot\rangle_{\bN^{-1}})$ and $(\mathbb{R}^{m}, \langle\cdot,\cdot\rangle_{\bM^{-1}})$, which are equipped with the $\bN^{-1}$- and $\bM^{-1}$-inner products, we introduce an iterative process that can generate a series of $\bN^{-1}$-orthonormal bases of $k$-dimensional subspaces of $(\mathbb{R}^{n}, \langle\cdot,\cdot\rangle_{\bN^{-1}})$. This process is mathematically equivalent to gen-GKB, thereby we use the same name for it. The use of the new inner products instead of the standard ones provides a very natural view for developing iterative regularization methods for Bayesian linear inverse problems.
	\item[$\bullet $] With the solution subspace spanned by the $\bN^{-1}$-orthonormal basis generated by gen-GKB, at the $k$-th step the SPR method reduces the original problem to a $k$-dimensional standard linear least squares problem. We further design an efficient procedure that can quickly update the iterative solution $x_k$ without solving the $k$-dimensional problem directly; the residual norm $\|\bA\bx_k-\bb\|_{\bM^{-1}}$ and solution norm $\|\bx\|_{\bN^{-1}}$ can also be updated quickly without computing $\bM^{-1}$ and $\bN^{-1}$. Using the residual norm or solution norm, we design three early stopping rules based on DP, LC and GCV, respectively. The most computationally intensive operations in the proposed algorithm only involve matrix-vector, thereby it is very efficient for large-scale problems.
	\item[$\bullet $] Several theoretical results related to the regularization properties of the proposed algorithm are established. Using the generalized singular value decomposition (GSVD) of $\{\bL_{M}\bA, \bL_{N}\}$, we give an explicit expression of the solution subspace and build connections between it and the best solution subspace for the SPR method (i.e. the subspace spanned by the dominant right generalized singular vectors). We prove that the iterative solution has a filtered GSVD expansion form, where some dominant GSVD components are captured and the others are filtered out. All these results reveal that the solution subspace can incorporate the information about the Gaussian noise $\bepsilon$ and desired property of $\bx$ encoded by the Gaussian prior, thereby the proposed algorithm can get a good regularized solution with similar accuracy as the best Tikhonov solution to \eqref{Bayes1}.
\end{itemize}

We use both small-scale and large-scale inverse problems to test the proposed algorithm and compare it with the generalized hybrid method in \cite{Chung2017}. The experimental results show that our algorithm is more robust and efficient for large-scale Bayesian linear inverse problems.

This paper is organized as follows. In Section \ref{sec2}, we give the expression of the solution to \eqref{Bayes1} by using GSVD of $\{\bL_{M}\bA, \bL_{N}\}$. In Section \ref{sec3}, we derive the gen-GKB process by treating $\bA$ as a bounded linear operator between the two Hilbert spaces $(\mathbb{R}^{n}, \langle\cdot,\cdot\rangle_{\bN^{-1}})$ and $(\mathbb{R}^{m}, \langle\cdot,\cdot\rangle_{\bM^{-1}})$, and then propose the gen-GKB based SPR algorithm with proper early stopping rules. In Section \ref{sec4}, we establish several theoretical results to reveal the regularization properties of the proposed algorithm. Numerical results are presented in Section \ref{sec5}, and conclusions are provided in Section \ref{sec6}.

\section{Solving the Bayesian linear inverse problem by GSVD} \label{sec2}
For analyzing regularized inverse problems, GSVD is a useful tool \cite{Hansen1989}. Since \eqref{gen_regu} and \eqref{Bayes1} are equivalent, we use the GSVD of the matrix pair $\{\bL_{M}\bA, \bL_{N}\}$ to discuss some properties of the MAP estimator that is the solution to \eqref{gen_regu}. Since $\bL_{N}$ is invertible, the GSVD of $\{\bL_{M}\bA, \bL_{N}\}$ can be written as
\begin{equation}\label{gsvd}
	\bL_{M}\bA = \bU_A\bSigma_A\bZ^{-1}, \ \ 
	\bL_{N} = \bU_L\bSigma_L\bZ^{-1},
\end{equation}
with
\begin{equation}\label{gsvd1}
	\bSigma_A = \bordermatrix*[()]{%
		\bD_{A}  &  &  r \cr
		&  \boldsymbol{0} & m-r \cr
		r & n-r 
	} , \ \ \
	\bSigma_L =
	\bordermatrix*[()]{%
		\bD_{L}  &  & r \cr
		&  \bI  & n-r \cr
		r & n-r
	} ,
\end{equation}
where $\bD_{A}=\mathrm{diag}(\sigma_1,\dots,\sigma_r)$, $1>\sigma_1\geq\cdots\geq\sigma_r>0$ and $\bD_{L}=\mathrm{diag}(\mu_1,\dots,\mu_r)$, $0<\mu_1\leq\cdots\leq\mu_r<1$, and it holds that $\sigma_{i}^{2}+\mu_{i}^{2}=1$. Throughout the paper, we denote by $\bI$ and $\boldsymbol{0}$ the identity matrix and zero matrix/vector, respectively, with orders clear from the context. The two matrices $\bU_{A}\in\mathbb{R}^{m\times m}$ and $\bU_{L}\in\mathbb{R}^{n\times n}$ are orthogonal, and $\bZ$ is invertible. Let $\bZ=(\bz_1,\dots,\bz_n)$. Define $\gamma_i=\sigma_i/\mu_i$ for $1\leq i\leq r$ and $\gamma_i=0$ for $r+1\leq i\leq n$. We call $\gamma_i$ the $i$-th largest generalized singular value and $\bz_i$ the $i$-th leading right generalized singular vector. Note that
\begin{equation}\label{gen_eig}
	\bZ^{\top}(\bA^{\top}\bM^{-1}\bA)\bZ = \bSigma_{A}^{\top}\bSigma_{A}, \ \ \ 
	\bZ^{\top}\bN^{-1}\bZ = \bSigma_{L}^2.
\end{equation}
Therefore, $\gamma_{i}^2$ are generalized eigenvalues of the matrix pair $\{\bA^{\top}\bM^{-1}\bA, \bN^{-1}\}$ with $\bz_i$ the corresponding generalized eigenvectors \cite[\S 8.7]{Golub2013}. Since $\bN^{-1}$ is positive definite, we can always normalize $\bz_i$ such that they are mutually $\bN^{-1}$-orthonormal, i.e. $\langle\bz_i,\bz_j\rangle_{\bN^{-1}}=\delta_{ij}$, where $\delta_{ij}$ is the Kronecker symbol and $\langle\bz,\bz'\rangle_{\bN^{-1}}:=\bz^{\top}\bN^{-1}\bz'$ is the $\bN^{-1}$-inner product defined in $\mathbb{R}^{n}$. In the following part, we always assume that $\bz_i$ are mutually $\bN^{-1}$-orthonormal. Note that in this case, $\bSigma_{L}=\bI$ and $\gamma_i=\sigma_i$ for $1\leq i\leq r$.

With the help of GSVD, the solution to \eqref{gen_regu} can be expressed as
\begin{align}\label{Tikh_gsvd}
	\bx_{\lambda} 
    &= [(\bL_{M}\bA)^{\top}\bL_{M}\bA+\lambda\bL_{N}^{\top}\bL_{N}]^{-1}(\bL_{M}\bA)^{\top}\bL_{M}\bb \nonumber \\
	&=\sum_{i=1}^{r}\frac{\gamma_{i}^{2}}{\gamma_{i}^{2}+\lambda}\frac{\bu_{A,i}^{\top}\bL_{M}\bb}{\sigma_i}\bz_i
\end{align}
where $\bu_{A,i}$ is the $i$-th column of $\bU_A$. This is just the filtered GSVD expansion form solution for the MAP estimator $\bx_{\lambda}$, where the filter factors $f_{i}:=\gamma_{i}^{2}/(\gamma_{i}^{2}+\lambda)$ are applied to noisy coefficients $\bu_{A,i}^{\top}\bL_{M}\bb$ to dampen noisy components appeared in the regularized solution \cite{Hansen1989}. The discrete Picard condition (DPC) plays a central role in the regularization of linear inverse problems, which says that the Fourier coefficients $|\bu_{A,i}^{\top}\bb_{\text{true}}|$ on average decay to zero faster than the corresponding $\gamma_{i}$, where $\bb_{\text{true}}=\bA\bx_{\text{true}}$ is the true observation \cite{hansen1990discrete}. Based on DPC, the regularization parameter $\lambda$ should be carefully chosen such that $f_{i}\approx 1$ for small $i$ and $f_{i}\approx 0$ for large $i$. 

To avoid determining $\lambda$ in advance, the most natural alternative approach is to directly let $f_i=1$ for small index and $f_i=0$ for large index, which leads to the so called truncated GSVD (TGSVD) solution
\begin{equation}\label{2.17}
	x_{k}^{\text{tgsvd}} = \sum_{i=1}^{k}\frac{\bu_{A,i}^{\top}\bL_{M}\bb}{\sigma_i}\bz_i.
\end{equation}
By truncating the TGSVD solution at a proper $k$, the obtained solution can capture main information corresponding to dominant right generalized singular vectors while suppressing noise corresponding to others \cite{Hansen1989}. These dominant $\bz_i$ play an important role in the regularized solution, since they contain the desirable information about $\bx$ encoded by the Gaussian prior. 

For large-scale Bayesian inverse problems, the Cholesky factorizations of $\bM^{-1}$ and $\bN^{-1}$ and the GSVD are very hard to compute. In this case, the generalized hybrid iterative method proposed in \cite{Chung2017} is a good choice, which first transforms \eqref{Bayes1} to \eqref{trans_tikh} and then uses gen-GKB to project \eqref{Bayes1} onto low dimensional subspaces to get small-scale problems. Although this method is effective for many types of inverse problems, sometimes the regularization parameter $\lambda_k$ can not be well determined, resulting in that $\lambda_k$ does not converge to $\lambda_{opt}$ or converges very slowly, and the corresponding iterative solution can only have a poor accuracy.

In the following part, we approach the Bayesian linear inverse problem from the perspective of subspace projection regularization instead of Tikhonov regularization. This perspective provides a direct motivation to derive our new iterative process for generating solution subspaces. Based on this iterative process, we propose the new iterative regularization algorithm.

\section{Subspace projection regularization based on generalized Golub-Kahan bidiagonalization} \label{sec3}
The \textit{subspace projection regularization} (SPR) is a unified framework for developing iterative regularization methods \cite[\S 3.3]{Engl2000}. For the Bayesian linear inverse problem \eqref{Bayes1}, the core idea is to compute a series of $\bx_k$ as the solution to
\begin{equation}\label{spr}
	\min_{\bx\in\mathcal{X}_k}\|\bx\|_{\bN^{-1}}, \ \ \mathcal{X}_k = \{\bx: \min_{\bx\in\mathcal{S}_k}\|\bA\bx-\bb\|_{\bM^{-1}}\} ,
\end{equation}
where $\mathcal{S}_k$ is the subspace of $\mathbb{R}^{n}$ of dimension $k=1,2,\dots$, and the iteration should be early stopped to overcome containing too many noisy components into the solution.

\begin{remark1} \label{rem_tgsvd} 
	The TGSVD solution $x_{k}^{\text{tgsvd}}$ is the $k$-th SPR solution if we choose $\mathcal{S}_k=\mathrm{span}\{\bZ_k\}$ for $k=1,2,\dots,r$, where $\bZ_k=(\bz_1,\dots,\bz_k)$. Throughout the paper, we denote by $\mathrm{span}\{\cdot\}$ the subspace spanned by columns of a matrix or a group of vectors. To see it, let $\bx=\bZ_k\by$ to denote any vector in $\mathrm{span}\{\bZ_k\}$, where $\by\in\mathbb{R}^{k}$. Then we have
	\begin{align*}
		\min_{\bx\in\mathcal{S}_k}\|\bA\bx-\bb\|_{\bM^{-1}} 
		&= \min_{\bx=\bZ_ky}\|\bL_{M}(\bA\bx-\bb)\|_{2} \\
		&= \min_{\by\in\mathbb{R}^{k}}\|\bL_{M}\bA\bZ_ky-\bL_{M}\bb\|_{2} \\
		&= \min_{\by\in\mathbb{R}^{k}}\|\bU_{A,k}\bSigma_{A,k}y-\bL_{M}\bb\|_{2} ,
	\end{align*}
	where $\bU_{A,k}$ is the first $k$ part of $\bU_{A}$ and $\bSigma_{A,k}$ is the first $k$-by-$k$ part of $\bSigma$, and the last equality uses the GSVD relation \eqref{gsvd}. The above least squares problem has a unique solution $\by_k=\bSigma_{A,k}^{-1}\bU_{A,k}^{\top}\bL_{M}\bb$. Thus, the solution to \eqref{spr} is $\bx_k=\bZ_k\by_k=\sum_{i=1}^{k}\frac{\bu_{A,i}^{\top}\bL_{M}\bb}{\sigma_i}\bz_i$, which is just $x_{k}^{\text{tgsvd}}$.
\end{remark1}

In fact, for any choice of $\mathcal{S}_k$, there exists a unique solution to \eqref{spr}. Using the same method as the above, let $\bx=\bW_k\by$ be any vector in $\mathcal{S}_k$, where $\bW_k\in\mathbb{R}^{n\times k}$ whose columns are $\bN^{-1}$-orthonormal and span $\mathcal{S}_k$. Then the solution to \eqref{spr} satisfies $\bx_k=\bW_k\by_k$, where $\by_k$ is the solution to
\begin{equation*}
	\min_{\by\in\mathcal{Y}_k}\|\by\|_{2}, \ \ \mathcal{Y}_k = \{\bx: \min_{\by\in\mathbb{R}^{k}}\|\bL_{M}\bA\bW_k\by-\bL_{M}\bb\|_{2}\} .
\end{equation*}
This problem has a unique solution $\by_k=(\bL_{M}\bA\bW_k)^{\dag}\bL_{M}\bb$, where $\bC^{\dag}$ is the pseudoinverse of a matrix $\bC$. Therefore, there exits a unique $\bx_k$.

Comparing to computing the MAP estimator $\bx_{\lambda}$ in the Tikhonov regularization framework \eqref{Bayes1}, the SPR framework utilizes those elaborately constructed subspaces $\mathcal{S}_k$ to enforce additional structure on $\bx$, which plays a similar role as the Tikhonov regularization term $\lambda\|\bx\|_{\bN^{-1}}$ arising from the Gaussian prior of $\bx$. We call $\mathcal{S}_k$ the \textit{solution subspace}, which should be constructed such that the prior information about the desired solution is incorporated into it. Based on the expressions of $\bx_{\lambda}$ and $x_{k}^{\text{tgsvd}}$, the most ideal choice is $\mathcal{S}_k=\mathrm{span}\{\bZ_k\}$, which just leads to the TGSVD solution.

\subsection{Generalized Golub-Kahan bidiagonalization}
For large-scale problems, computing GSVD is impractical. A reasonable SPR method should construct a series of solution subspaces that capture main information corresponding to some leading right generalized singular vectors $\{\bz_i\}_{i=1}^{k}$ as the iteration proceeds. Note that $\bz_i$ are mutually $\bN^{-1}$-orthonormal, and note that \eqref{spr} has the form of weighted least squares. This motivates us to design an iterative process that generates a series of $\bN^{-1}$-orthonormal vectors and projects \eqref{spr} onto subspaces spanned by these vectors.

\begin{definition}\label{def:inner_prod}
	Define two finite dimensional Hilbert spaces as $(\mathbb{R}^{n}, \langle\cdot,\cdot\rangle_{\bN^{-1}})$ and $(\mathbb{R}^{m}, \langle\cdot,\cdot\rangle_{\bM^{-1}})$ that are equipped with the $\bN^{-1}$-inner product $\langle\bx,\bx'\rangle_{\bN^{-1}}:=\bx^{\top}\bN^{-1}\bx'$ and $\bM^{-1}$-inner product $\langle\by,\by'\rangle_{\bM^{-1}}:=\by^{\top}\bM^{-1}\by'$, respectively. 
\end{definition}

Although the two topologies of $\bN^{-1}$- and standard inner product spaces are equivalent, the geometry structures of these two spaces are different. For dealing with \eqref{spr}, using $\bN^{-1}$- and $\bM^{-1}$-inner product is more natural. In order to develop iterative process for $\bA$ between the two spaces $(\mathbb{R}^{n}, \langle\cdot,\cdot\rangle_{\bN^{-1}})$ and $(\mathbb{R}^{m}, \langle\cdot,\cdot\rangle_{\bM^{-1}})$, we first need the adjoint of the linear operator $\bA$.
	For the linear operator $\bA: (\mathbb{R}^{n}, \langle\cdot,\cdot\rangle_{\bN^{-1}}) \to (\mathbb{R}^{m}, \langle\cdot,\cdot\rangle_{\bM^{-1}})$ defined as $\bA: \bx\mapsto \bA\bx$ for $\bx\in\mathbb{R}^{n}$ under the canonical bases of $\mathbb{R}^{n}$ and $\mathbb{R}^{m}$, define its adjoint $\bA^{*}: (\mathbb{R}^{m}, \langle\cdot,\cdot\rangle_{\bM^{-1}}) \to (\mathbb{R}^{n}, \langle\cdot,\cdot\rangle_{\bN^{-1}})$ as $\bA^{*}: \by \mapsto \bA^{*}\by$ for $\by\in\mathbb{R}^{m}$, by the relation
	\begin{equation*}
		\langle \bA\bx, \by\rangle_{\bM^{-1}} = \langle \bx, \bA^{*}\by\rangle_{\bN^{-1}} .
	\end{equation*}
Since the two Hilbert spaces are finite dimensional, it follows that $\bA$ is bounded, and the adjoint operator $\bA^{*}$ is a well-defined linear operator. Under the above two definitions, we can develop an iterative process that reduces $\bA$ to a bidiagonal form and generates two subspaces of $(\mathbb{R}^{n}, \langle\cdot,\cdot\rangle_{\bN^{-1}})$ and $(\mathbb{R}^{m}, \langle\cdot,\cdot\rangle_{\bM^{-1}})$, respectively. This can be achieved by applying the GKB process to $\bA$ with starting vector $\bb$ between the two Hilbert spaces $(\mathbb{R}^{n}, \langle\cdot,\cdot\rangle_{\bN^{-1}})$ and $(\mathbb{R}^{m}, \langle\cdot,\cdot\rangle_{\bM^{-1}})$; for the GKB process for compact linear operator, see \cite{caruso2019convergence}. The basic recursive relations are as follows:
\begin{align}
	& \beta_1\bu_1 = \bb, \label{GKB1} \\
	& \alpha_{i}\bv_i = \bA^{*}\bu_i -\beta_i\bv_{i-1}, \label{GKB2} \\
	& \beta_{i+1}\bu_{i+1} = \bA\bv_{i} - \alpha_i\bu_i, \label{GKB3}
\end{align}
where $\alpha_i$ and $\beta_i$ are computed such that $\|\bu_i\|_{\bM^{-1}}=\|\bv_i\|_{\bN^{-1}}=1$, and we let $\bv_0=\boldsymbol{0}$. In order to compute $\bA^{*}\bu_i$, we need the matrix-form expression of $\bA^{*}$.

\begin{lemma}\label{lem:adj_matform}
	Under the canonical bases of $\mathbb{R}^{n}$ and $\mathbb{R}^{m}$, we have
	\begin{equation}\label{adjoint}
		\bA^{*} = \bN\bA^{\top}\bM^{-1}.
	\end{equation}
\end{lemma}
\begin{proof}
	Under the canonical bases, the linear mapping $\bA\bx$ and $\bA^{*}\by$ are the standard matrix-vector products. By the definition of $\bA^{*}$, we have
	\[(\bA\bx)^{\top}\bM^{-1}\by = \bx^{\top}\bN^{-1}\bA^{*}\by \]
	for any $\bx\in\mathbb{R}^{n}$ and $\by\in\mathbb{R}^{m}$. Therefore we have $\bA^{\top}\bM^{-1}=\bN^{-1}\bA^{*}$, and the desired result is immediately obtained. \qed
\end{proof}

In order to avoid explicitly computing $\bN^{-1}$-norm to obtain $\alpha_i$, let $\bar{\bu}_i=\bM^{-1}\bu_i$ and $\bar{\bv}_i=\bN^{-1}\bv_i$. By \eqref{GKB2} and \eqref{adjoint}, we have
\begin{align}
	\alpha_i\bar{\bv}_{i} = \bA^{\top}\bar{\bu}_i - \beta_i\bar{\bv}_{i-1}.
\end{align}
Let $\bar{\br}_i= A^{\top}\bar{\bu}_i - \beta_i\bar{\bv}_{i-1}$ and $\bs_i=\bA\bv_{i} - \alpha_i\bu_{i}$. Then $\alpha_i=\|\bN\bar{\br}_i\|_{\bN^{-1}}=(\bar{\br}_{i}^{\top}\bN\bar{\br}_i)^{1/2}$ and $\beta_{i+1}=\|\bs_i\|_{\bM^{-1}}=(\bs_{i}^{\top}\bM^{-1}\bs_i)^{1/2}$. We remark that computing with $\bM^{-1}$ can not be avoided to get $\bu_i$. For the most commonly encountered cases that $\bepsilon$ is a Gaussian white noise or all of its components are uncorrelated, $\bM$ is diagonal and $\bM^{-1}$ can be directly obtained. 
The whole iterative process is summarized in Algorithm \ref{alg1}.

\begin{algorithm}[htb]
	\caption{Generalized Golub-Kahan bidiagonalization (gen-GKB)}\label{alg1}
	\begin{algorithmic}[1]
		\Require $\bA\in\mathbb{R}^{m\times n}$, $\bb\in\mathbb{R}^{m}$, $\bM\in\mathbb{R}^{m\times m}$, $\bN\in\mathbb{R}^{n\times n}$
		\State $\bar{\bs}=\bM^{-1}\bb$ 
		\State $\beta_1=\bar{\bs}^{\top}\bb$, \ $\bu_1=\bb/\beta_1$, \ $\bar{\bu}_1=\bar{\bs}/\beta_1$
		\State $\bar{\br}=\bA^{\top}\bar{\bu}_1$, \ $\br=\bN\bar{\br}$
		\State $\alpha_1 = (\br^{\top}\bar{\br})^{1/2}$, \ $\bar{\bv}_1=\bar{\br}/\alpha_1$, \ $\bv_1=\br/\alpha_1$
		\For {$i=1,2,\dots,k$}
		\State $\bs=\bA\bv_{i} - \alpha_i\bu_{i}$ 
		\State $\bar{\bs}=\bM^{-1}\bs$ 
		\State $\beta_{i+1}=(\bs^{\top}\bar{\bs})^{1/2}$, \ $\bu_{i+1}=\bs/\beta_{i+1}$, \ $\bar{\bu}_{i+1}=\bar{\bs}/\beta_{i+1}$ \Comment{$\bar{\bu}_{i+1}=\bM^{-1}\bu_{i+1}$}
		\State $\bar{\br}=\bA^{\top}\bar{\bu}_{i+1} - \beta_{i+1}\bar{\bv}_{i}$
		\State $\br=\bN\bar{\br}$ 
		\State $\alpha_{i+1}=(\br^{\top}\bar{\br})^{1/2}$, \ $\bar{\bv}_{i+1}=\bar{\br}/\alpha_{i+1}$, \ $\bv_{i+1}=\br/\alpha_{i+1}$  \Comment{$\bar{\bv}_{i+1}=\bN^{-1}\bv_{i+1}$}
		\EndFor
		\Ensure $\{\alpha_i, \beta_i\}_{i=1}^{k+1}$, \ $\{\bu_i, \bv_i\}_{i=1}^{k+1}$, \ $\{\bar{\bu}_i, \bar{\bv}_i\}_{i=1}^{k+1}$
	\end{algorithmic}
\end{algorithm}

One can verify that the gen-GKB derived here is equivalent to the iterative process used in \cite{Chung2017} with the same name. However, here we adopt a new point of view, i.e. treating $\bA$ as a linear operator between two Hilbert spaces $(\mathbb{R}^{n}, \langle\cdot,\cdot\rangle_{\bN^{-1}})$ and $(\mathbb{R}^{m}, \langle\cdot,\cdot\rangle_{\bM^{-1}})$ and applying GKB to $\bA$ between these two spaces to get gen-GKB. This view is particularly intuitive for constructing valid solution subspaces in SPR methods for the Bayesian inverse problems. Based on this new view, we gain insight into the structures of the two subspaces $\mathrm{span}\{\bu_i\}_{i=1}^{k}$ and $\mathrm{span}\{\bv_i\}_{i=1}^{k}$.

\begin{proposition}\label{prop:Krylov}
	The group of vectors $\{\bu_i\}_{i=1}^{k}$ is an $\bM^{-1}$-orthonormal basis of the Krylov subspace 
	\begin{equation*}
		\mathcal{K}_k(\bA\bN\bA^{\top}\bM^{-1}, b) = \mathrm{span}\{(\bA\bN\bA^{\top}\bM^{-1})^ib\}_{i=0}^{k-1},
	\end{equation*}
	and $\{\bv_i\}_{i=1}^{k}$ is an $\bN^{-1}$-orthonormal basis of the Krylov subspace 
	\begin{equation*}
		\mathcal{K}_k(\bN\bA^{\top}\bM^{-1}\bA, \bN\bA^{\top}\bM^{-1}\bb) = \mathrm{span}\{(\bN\bA^{\top}\bM^{-1}\bA)^i\bN\bA^{\top}\bM^{-1}\bb\}_{i=0}^{k-1}.
	\end{equation*}
\end{proposition}
\begin{proof}
	We exploit the property of GKB for the linear operator $\bA$ with starting vector $b$ between the two Hilbert spaces $(\mathbb{R}^{n}, \langle\cdot,\cdot\rangle_{\bN^{-1}})$ and $(\mathbb{R}^{m}, \langle\cdot,\cdot\rangle_{\bM^{-1}})$, which states that $\{\bu_i\}_{i=1}^{k}$ and $\{\bv_i\}_{i=1}^{k}$ are the $\bM^{-1}$-orthonormal basis and $\bN^{-1}$-orthonormal basis of the Krylov subspaces $\mathcal{K}_k(\bA\bA^{*}, \bb)$ and $\mathcal{K}_k(\bA^{*}\bA, \bA^{*}\bb)$, respectively. Using Lemma \ref{lem:adj_matform}, we have $\bA\bA^{*}=\bA\bN\bA^{\top}\bM^{-1}$, $\bA^{*}\bA=\bN\bA^{\top}\bM^{-1}\bA$ and $\bA^{*}\bb=\bN\bA^{\top}\bM^{-1}\bb$. The proof is completed. \qed
\end{proof}

Define two matrices $\bU_k=(\bu_1,\dots,\bu_k)$ and $\bV_{k}=(\bv_1,\dots,\bv_k)$. Then $\bU_k$ and $\bV_k$ are two $\bM^{-1}$ and $\bN^{-1}$ orthonormal matrices, respectively. By \eqref{GKB1}--\eqref{GKB3}, we can rewrite gen-GKB with the matrix-form recursive relations:
\begin{align}
	& \beta_1\bU_{k+1}e_{1} = \bb, \label{GKB13} \\
	& \bA\bV_k = \bU_{k+1}\bB_k, \label{GKB23} \\
	& \bN\bA^{\top}\bM^{-1}\bU_{k+1} = \bV_k\bB_{k}^{\top}+\alpha_{k+1}\bv_{k+1}e_{k+1}^{\top}, \label{GKB33}
\end{align}
where $\be_1$ and $\be_{k+1}$ are the first and $(k+1)$-th columns of the identity matrix of order $k+1$, respectively, and 
\begin{equation}
	\bB_{k}=\begin{pmatrix}
		\alpha_{1} & & & \\
		\beta_{2} &\alpha_{2} & & \\
		&\beta_{3} &\ddots & \\
		& &\ddots &\alpha_{k} \\
		& & &\beta_{k+1}
	\end{pmatrix}\in  \mathbb{R}^{(k+1)\times k}
\end{equation}
is the projection of $\bA$ onto subspaces $\mathrm{span}\{\bU_{k+1}\}$ and $\mathrm{span}\{\bV_k\}$. If gen-GKB does not terminate in $k$ steps, i.e. $\alpha_i, \beta_i\neq0$ for $1\leq i \leq k+1$, then $B_k$ has full column rank. Based on the above matrix-form relations, we can design the iterative SPR algorithm based on gen-GKB.

\subsection{Iterative regularization by subspace projection}
At the $k$-th iteration of gen-GKB, we seek solution to \eqref{spr} by setting $\mathcal{S}_k=\mathrm{span}\{\bV_k\}$. Denote any vector $\bx\in \mathcal{S}_k$ as $\bx=\bV_{k}\by$ with $\by\in\mathbb{R}^{k}$. By \eqref{GKB23}, we have
\begin{align*}
	\min_{x\in\mathcal{S}_k}\|\bA\bx-\bb\|_{\bM^{-1}}
	&= \min_{y\in\mathbb{R}^{k}}\|\bA\bV_ky-\bU_{k+1}\beta_1\be_{1}\|_{\bM^{-1}} \\
	&= \min_{y\in\mathbb{R}^{k}}\|\bU_{k+1}(\bB_k\by-\bU_{k+1}\beta_1e_{1})\|_{\bM^{-1}} \\
	&= \min_{y\in\mathbb{R}^{k}}\|\bB_k\by-\beta_1\be_{1}\|_2,
\end{align*}
where we use the property that $\bU_{k+1}$ is $\bM^{-1}$-orthonormal. Under the assumption that gen-GKB does not terminate, 
the above $k$-dimensional linear least squares problem has a unique solution. Therefore, the solution to \eqref{spr} is
\begin{equation}\label{k_dim_ls}
	\bx_k = \bV_k\by_k, \ \ \ \by_k=\argmin_{y\in\mathbb{R}^{k}}\|\bB_k\by-\beta_1e_{1}\|_2
\end{equation}
for $k=1,2\dots$. Furthermore, we do not need to explicitly compute $\by_k$ by solving \eqref{k_dim_ls} for each $k$. Instead, there is an efficient procedure to update $\bx_k$ from $\bx_0=\mathbf{0}$ step by step. The procedure is similar to the LSQR algorithm for standard linear least squares problems \cite{Paige1982}.

First, perform the following QR factorization: 
\[\widetilde{\bQ}_k\begin{pmatrix}
	\bB_k & \beta_{1}\be_{1}
\end{pmatrix}
=\begin{pmatrix}
	\bR_k & \bg_k \\
	& \bar{\phi}_{k+1}
\end{pmatrix}=
\left(\begin{array}{ccccc:c}		
	\rho_1& \theta_2 &  &  &  & \phi_1 \\		
	& \rho_2 & \theta_3 & & & \phi_2 \\		
	& & \ddots & \ddots & & \vdots \\ 		
	& & & \rho_{k-1} & \theta_k & \phi_{k-1} \\		
	& & &  & \rho_k & \phi_k \\	\hdashline	
	& & &  &  & \bar{\phi}_{k+1}	
\end{array}\right) ,\]
where the orthogonal matrix $\widetilde{\bQ}_k$ is the product of a series of Givens rotation matrices that is used to zero out $\beta_{i+1}$:
\[\begin{pmatrix}
	c_i & s_i \\
	s_i & -c_i
\end{pmatrix}
\begin{pmatrix}
	\bar{\rho}_i & 0 & \bar{\phi}_i  \\
	\beta_{i+1} & \alpha_{i+1} & 0
\end{pmatrix}=
\begin{pmatrix}
	\rho_i & \theta_{i+1} & \phi_i \\
	0 & \bar{\rho}_{i+1}  & \bar{\phi}_{i+1}
\end{pmatrix} , \]
where $i=1,\dots,k$. There exist efficient recursive formulae to compute these matrix elements, and here we adopt these formulae proposed in \cite{Paige1982} (line 3--6 in Algorithm \ref{up_alg}). Using the orthogonality of $\widetilde{\bQ}_k$ leads to
\begin{equation}\label{res1}
	\|\bB_k\by-\beta_1e_{1}\|_{2}^2=\left\|\widetilde{\bQ}_k\begin{pmatrix}
		\bB_k & \beta_{1}\be_{1}
	\end{pmatrix}\begin{pmatrix}
		\by \\ -1
	\end{pmatrix}\right\|_{2}^2=\|\bR_k\by-\bg_k\|_{2}^2+|\bar{\phi}_{k+1}|^2 .
\end{equation}
Thereby the solution to $\argmin_{\by\in\mathbb{R}^k}\|\bB_k\by-\beta_{1}\be_{1}\|_2$ is $\by_k=\bR_{k}^{-1}\bg_k$. Factorizing $\bR_k$ as
\[\bR_k=\bD_k\bar{\bR}_k, \ \ 
\bD_k=\begin{pmatrix}
	\rho_1 & & & \\
	& \rho_2 & & \\
	& & \ddots & \\
	& & & \rho_k
\end{pmatrix}, \ \
\bar{\bR}_k=\begin{pmatrix}
	1 & \theta_2/\rho_1 & & \\
	& 1 & \theta_3/\rho_2 & \\
	& & \ddots & \theta_k/\rho_{k-1} \\
	& & & 1
\end{pmatrix},\]
we obtain
\begin{equation}\label{up1_x}
	\bx_k=\bV_k\by_k=\bV_k\bR_{k}^{-1}\bg_k=(\bV_k\bar{\bR}_{k}^{-1})(\bD_{k}^{-1}\bg_k).
\end{equation}
Let $\bW_k=\bV_k\bar{\bR}_{k}^{-1}=(\bw_1,\dots,\bw_k)$. Note that $\bD_{k}^{-1}\bg_k=(\phi_{1}/\rho_{1},\dots,\phi_{k}/\rho_{k})^{\top}$. Solving $\bW_k\bar{\bR}_k=\bV_k$ by back substitution and combining \eqref{up1_x} leads to the recursive formula for $\bx_i$ and $\bw_i$: 
\begin{equation}\label{up2_x}
	\bx_{i}=\bx_{i-1}+(\phi_{i}/\rho_{i})\bw_{i}, \ \ \  
	\bw_{i+1}=\bv_{i+1}-(\theta_{i+1}/\rho_{i})\bw_{i}.
\end{equation}

From the above derivation, we have 
\begin{equation}\label{res_norm}
	\bar{\phi}_{k+1}=\|\bB_k\by_{k}-\beta_1\be_{1}\|_2 = \|\bA\bx_k-\bb\|_{\bM^{-1}}.
\end{equation}
Thereby the residual norm can be updated very quickly without explicitly computing $\|\bA\bx_k-\bb\|_{\bM^{-1}}$. In order to compute $\|\bx_k\|_{\bN^{-1}}$ without explicitly using $\bN^{-1}$, let $\bar{\bx}_k=\bN^{-1}\bx_k$ and $\bar{\bw}_k=\bN^{-1}\bw_k$. By \eqref{up2_x} we have
\begin{equation}\label{up3_x}
	\bar{\bx}_{i}=\bar{\bx}_{i-1}+(\phi_{i}/\rho_{i})\bar{\bw}_{i}, \ \ \  
	\bar{\bw}_{i+1}=\bar{\bv}_{i+1}-(\theta_{i+1}/\rho_{i})\bar{\bw}_{i},
\end{equation}
and $\|\bx_k\|_{\bN^{-1}}=(\bx_{k}^{\top}\bar{\bx}_k)^{1/2}$. The above updating procedure is summarized in Algorithm \ref{up_alg}.

\begin{algorithm}[htb]
	\caption{Procedure for updating iterative solutions and residual/solution norms}\label{up_alg}
	\begin{algorithmic}[1]
		\State {Let $\bx_{0}=\mathbf{0},\ \bar{\bx}_0=\mathbf{0}, \ \bw_{1}=\bv_{1}, \ \bar{\bw}_{1}=\bar{\bv}_1, \ \bar{\phi}_{1}=\beta_{1},\ \bar{\rho}_{1}=\alpha_{1}$}
		\For{$i=1,2,\dots,k$}
		\State $\rho_{i}=(\bar{\rho}_{i}^{2}+\beta_{i+1}^{2})^{1/2}$
		\State $c_{i}=\bar{\rho}_{i}/\rho_{i},\ s_{i}=\beta_{i+1}/\rho_{i}$
		\State$\theta_{i+1}=s_{i}\alpha_{i+1},\ \bar{\rho}_{i+1}=-c_{i}\alpha_{i+1}$
		\State $\phi_{i}=c_{i}\bar{\phi}_{i},\ \bar{\phi}_{i+1}=s_{i}\bar{\phi}_{i} $
		\State $\bx_{i}=\bx_{i-1}+(\phi_{i}/\rho_{i})\bw_{i} $
		\State $\bw_{i+1}=\bv_{i+1}-(\theta_{i+1}/\rho_{i})\bw_{i}$
		\State $\bar{\bx}_{i}=\bar{\bx}_{i-1}+(\phi_{i}/\rho_{i})\bar{\bw}_{i}$  \Comment{$\bar{\bx}_{i}=\bN^{-1}\bx_{i}$}
		\State $\bar{\bw}_{i+1}=\bar{\bv}_{i+1}-(\theta_{i+1}/\rho_{i})\bar{\bw}_{i}$  \Comment{$\bar{\bw}_{i+1}=\bN^{-1}\bw_{i+1}$}
		\State $\|\bx_k\|_{\bN^{-1}}=(\bx_{k}^{\top}\bar{\bx}_k)^{1/2}$
		\EndFor
	\end{algorithmic}
\end{algorithm}

\subsection{Early stopping rules}
The SPR method usually exhibits semi-convergence behavior: as the iteration proceeds, the iterative solution first gradually approximate to $\bx_{\text{true}}$, then the solution will deviate far from $\bx_{\text{true}}$ and eventually converges to $\bx_{\text{naive}}=\bA^{\dag}\bb$ \cite[\S 3.3]{Engl2000}. This is because the solution subspace will contain more and more noisy components as it gradually expands. The iteration at which the corresponding solution has the smallest error is called the semi-convergence point. In order to obtain a good regularized solution, we should stop the iteration early. Note that the iteration number $k$ in SPR plays a similar role as the regularization parameter in Tikhonov regularization. Here we adapt several criteria for choosing regularization parameters to estimate the semi-convergence point.

\paragraph*{Discrepancy principle.} \ 
For the case that $\bepsilon$ is a Gaussian white noise, if $\|\bepsilon\|_2$ is known, one criterion for determining the early stopping iteration is the discrepancy principle (DP), which states that the discrepancy between the data and predicted output $\|\bA\bx_k-\bb\|_2$, should be of the order of $\|\bepsilon\|_2$ \cite[\S 7.2]{Hansen1998}. If $\bepsilon$ is a non-white Gaussian noise, noticing that \eqref{inverse1} leads to
\begin{equation}\label{trans}
	\bL_{M}\bb = \bL_{M}\bA\bx + \bL_{M}\bepsilon,
\end{equation}
and $\bL_{M}\bepsilon\sim\mathcal{N}(\mathbf{0}, \bI)$, thereby this transformation whitens the noise, and the discrepancy is $\|\bL_{M}\bA\bx_k-\bL_{M}\bb\|_2=\|\bA\bx_k-\bb\|_{\bM^{-1}}$. Since $\bar{\bepsilon}=\bL_{M}\bepsilon$ is a Gaussian white noise with zero mean, it follows that
\begin{equation*}
	\mathbb{E}\left[\|\bar{\bepsilon}\|_{2}^{2}\right] = \mathbb{E}\left[\mathrm{trace}\left(\bar{\bepsilon}^{\top}\bar{\bepsilon}\right)\right]
	= \mathbb{E}\left[\mathrm{trace}\left(\bar{\bepsilon}\bar{\bepsilon}^{\top}\right)\right]
	= \mathrm{trace}\left(\mathbb{E}\left[\bar{\bepsilon}\bar{\bepsilon}^{\top}\right]\right)
	= \mathrm{trace}\left(\bI\right) = m.
\end{equation*}
Note from \eqref{res_norm} that $\|\bA\bx_k-\bb\|_{\bM^{-1}}=\bar{\phi}_{k+1}$ decreases monotonically since $\bx_k$ minimizes $\|\bA\bx-\bb\|_{\bM^{-1}}$ in the gradually expanding subspace $\mathrm{span}\{\bV_k\}$. Following DP, we should stop iteration at the first $k$ satisfying 
\begin{equation}\label{discrepancy}
	\bar{\phi}_{k+1} \leq \tau\sqrt{m}
\end{equation}
with $\tau>1$ slightly, such as $\tau=1.01$. Typically, the early stopping iteration determined by DP is slightly smaller than the semi-convergence point, thereby the corresponding solution is over-regularized.

\paragraph*{L-curve.} \
Another early stopping rule is the L-curve criterion, which is a heuristic approach that does not need the noise norm \cite{Hansen1992}. The idea is to plot in log-log scale the following curve
\begin{equation}
	\left(\log\|\bA\bx_k-\bb\|_{\bM^{-1}},\log\|\bx_k\|_{\bN^{-1}}\right) 
	= \left(\log\bar{\phi}_{k+1},\log(\bx_{k}^{\top}\bar{\bx}_k)^{1/2}\right),
\end{equation}
which frequently has a characteristic `L' shape, and the corner corresponds to the point where further reduction in the residual comes only at the expense of a drastic increase in the regularization term. Thus, we should choose the early stopping iteration that corresponds to this corner, which is usually defined as the point of maximum curvature of the L-curve in a log-log plot.

\paragraph*{Generalized cross-validation.} \
The generalized cross-validation (GCV) is a statistical approach for estimating the optimal regularization parameter, especially for the Gaussian white noise case \cite{Golub1979}. For the SPR method where the iteration number $k$ plays the role of the regularization parameter, this approach can be adapted to estimate the optimal iteration number $k_0$. Since the transformed system \eqref{trans} equivalent to \eqref{inverse1} has a white Gaussian noise $\bL_{M}\bepsilon$, the GCV function with respect to $k$ should be
\begin{equation*}
	\mathrm{GCV}(k) = \frac{\|\bL_{M}\bA\bx_k-\bL_{M}\bb\|_{2}^{2}}{(\mathrm{trace}(\bI-\bL_{M}\bA\bA_{k}^{\dag}))^2} = \frac{\|\bA\bx_k-\bb\|_{\bM^{-1}}^{2}}{(\mathrm{trace}(\bI-\bL_{M}\bA\bA_{k}^{\dag}))^2},
\end{equation*} 
where $\bA_{k}^{\dag}$ denotes the matrix that maps the right-hand side $\bL_{M}b$ to the $k$-th regularized solution $\bx_k$. By \eqref{k_dim_ls} and \eqref{GKB13} we have
\begin{align*}
	\bx_k = \bV_{k}\bB_{k}^{\dag}\beta_{1}\be_{1}=\bV_{k}\bB_{k}^{\dag}\bU_{k+1}^{\top}\bM^{-1}\bb
	=(\bV_{k}\bB_{k}^{\dag}\bU_{k+1}^{\top}\bM^{-1}\bL_{M}^{-1})\bL_{M}b,
\end{align*}
which implies that $\bA_{k}^{\dag}=\bV_{k}\bB_{k}^{\dag}\bU_{k+1}^{\top}\bM^{-1}\bL_{M}^{-1}$. Using \eqref{GKB23} again we get 
\begin{align*}
	\mathrm{trace}(\bL_{M}\bA\bA_{k}^{\dag}) 
	&= \mathrm{trace}(\bL_{M}\bA\bV_{k}\bB_{k}^{\dag}\bU_{k+1}^{\top}\bM^{-1}\bL_{M}^{-1})
	= \mathrm{trace}(\bU_{k+1}\bB_{k}\bB_{k}^{\dag}\bU_{k+1}^{\top}\bM^{-1}) \\
	&= \mathrm{trace}(\bB_{k}\bB_{k}^{\dag}\bU_{k+1}^{\top}\bM^{-1}\bU_{k+1}) \\
	&= \mathrm{trace}(\bB_{k}\bB_{k}^{\dag})
	= k.
\end{align*}
Therefore, the final expression of the GCV function is
\begin{equation}
	\mathrm{GCV}(k) = \frac{\bar{\phi}_{k+1}^2}{(m-k)^2},
\end{equation}
and the minimizer of this function is used as a good estimate of $k_0$.

We remark that in using the GCV or L-curve method, one must go a few iterations beyond the optimal $k$ in order to verify the optimum. The gen-GKB based SPR iterative algorithm for \eqref{inverse1} is summarized in Algorithm \ref{alg3}.

\begin{algorithm}[htb]
	\caption{Subspace projection regularization by gen-GKB (genGKB\_SPR)}\label{alg3}
	\begin{algorithmic}[1]
		\Require $\bA\in\mathbb{R}^{m\times n}$, $\bb\in\mathbb{R}^{m}$, $\bM\in\mathbb{R}^{m\times m}$, $\bN\in\mathbb{R}^{n\times n}$
		\State Initialize $\bx_{0}=\mathbf{0}$, $\bar{\bx}_{0}=\mathbf{0}$, 
		\For{$k=1,2,\ldots,$}
		\State Compute $\bv_k$, $\bar{\bv}_k$, $\alpha_k$, $\beta_k$ by \texttt{gen-GKB} 
		\State Compute $\bx_k$, $\bar{\bx}_k$, $\bar{\phi}_{k+1}$, $\|\bx_k\|_{\bN^{-1}}$ by Algorithm \ref{up_alg} 
		\If{Early stopping criterion is satisfied}  
		\State Denote the estimated optimal iteration by $k_1$
		\State Terminate the iteration at $k_1$ 
		\EndIf
		\EndFor
		\Ensure Final regularized solution $\bx_{k_1}$  
	\end{algorithmic}
\end{algorithm}

\section{Regularization properties of genGKB\_SPR}\label{sec4}
We establish several theoretical results about the genGKB\_SPR algorithm to reveal its regularization properties for solving the Bayesian linear inverse problem. Based on the GSVD of $\{\bL_{M}\bA, \bL_{N}\}$ appeared from the transformed standard 2-norm Tikhonov regularization \eqref{gen_regu}, we first give an explicit expression of the solution subspaces $\mathrm{span}\{\bV_k\}$ and build connections with $\mathrm{span}\{\bZ_k\}$, which are the most ideal solution subspaces for the SPR method.

\begin{lemma}\label{Struct_Vk}
	Following the notations in \eqref{gsvd} and \eqref{gsvd1}, the $k$-th solution subspace of genGKB\_SPR is
	\begin{equation}
		\mathrm{span}\{\bV_k\} = \mathrm{span}\{\bZ_r\bD_{A}^{2i+1}\bU_{A,r}^{\top}\bL_{M}\bb\}_{i=0}^{k-1},
	\end{equation}
	where $\bZ_r$ and $\bU_{A,r}$ are the first columns of $\bZ$ and $\bU_{A}$, respectively.
\end{lemma} 
\begin{proof}
	By Proposition \ref{prop:Krylov} we have $\mathrm{span}\{\bV_k\}=\mathcal{K}_k(\bN\bA^{\top}\bM^{-1}\bA, \bN\bA^{\top}\bM^{-1}\bb)$. Using the GSVD of $\{\bL_{M}\bA, \bL_{N}\}$, we have
	\begin{align*}
		\bN\bA^{\top}\bM^{-1}\bA
		= \bN\bZ^{-\top}\bSigma_{A}^{\top}\bU_{A}^{\top}\bL_{M}^{-\top}\bM^{-1}\bL_{M}^{-1}\bU_{A}\bSigma_{A}\bZ^{-1} 
		=  \bZ\bSigma_{A}^{\top}\bSigma_{A}\bZ^{-1},
	\end{align*}
	where we use $\bN\bZ^{-\top}=\bZ$ since $\bZ^{\top}\bN^{-1}\bZ=\bI$. Similarly,
	\begin{equation*}
		\bN\bA^{\top}\bM^{-1}\bb = \bN\bZ^{-\top}\bSigma_{A}^{\top}\bU_{A}^{\top}\bL_{M}^{-\top}\bM^{-1}\bb
		= \bZ\bSigma_{A}^{\top}\bU_{A}^{\top}\bL_{M}\bb .
	\end{equation*}
	Therefore, for $i=0,1,\dots,k-1$, it holds
	\begin{align*}
		(\bN\bA^{\top}\bM^{-1}\bA)^{i}\bN\bA^{\top}\bM^{-1}\bb
		&= (\bZ\bSigma_{A}^{\top}\bSigma_{A}\bZ^{-1})^{i}\bZ\bSigma_{A}^{\top}\bU_{A}^{\top}\bL_{M}\bb \\
		&= \bZ(\bSigma_{A}^{\top}\bSigma_{A})^{i}\bSigma_{A}^{\top}\bU_{A}^{\top}\bL_{M}\bb \\
		&= \bZ_r\bD_{A}^{2i+1}\bU_{A,r}^{\top}\bL_{M}\bb.
	\end{align*}
	The desired result immediately follows. \qed
\end{proof}

To investigate how far the $k$-th solution subspace is from the ideal one $\mathrm{span}\{\bZ_k\}$, we use the following definition for measuring the distance between two subspaces under a non-standard inner product; see \cite[Chapter 4, \S 2.1]{Stewart2001matrix} and \cite[\S 2.2.5]{bjorck2015numerical} for the corresponding definition under the standard inner product.

\begin{definition}\label{car_angle}
	Let $\mathcal{F}$ and $\mathcal{G}$ be two $k$-dimensional subspaces of $(\mathbb{R}^{n}, \langle\cdot,\cdot\rangle_{\bN^{-1}})$. Define the maximum canonical angle between $\mathcal{F}$ and $\mathcal{G}$ as
	\begin{equation*}
		\Theta(\mathcal{F}, \mathcal{G}) = \max\limits_{\bu\in\mathcal{F} \atop \bu\neq\mathbf{0}}\min\limits_{\bv\in\mathcal{G} \atop \bv\neq\mathbf{0}}\arccos \frac{|\langle\bu,\bv\rangle_{\bN^{-1}}|}{\|\bu\|_{\bN^{-1}}\|\bv\|_{\bN^{-1}}}.
	\end{equation*}
	The distance between $\mathcal{F}$ and $\mathcal{G}$ is defined as
	\begin{equation*}
		\mathrm{dist}(\mathcal{F}, \mathcal{G}) = \sin\Theta(\mathcal{F}, \mathcal{G}) .
	\end{equation*}
\end{definition}

\begin{lemma}\label{angle_expres}
	Let $(\bF_1, \bF_2)$ and $(\bG_1, \bG_2)$ be two $\bN^{-1}$-orthogonal matrices, both of whose columns span $(\mathbb{R}^{n}, \langle\cdot,\cdot\rangle_{\bN^{-1}})$. If $\mathrm{span}\{\bF_1\}=\mathcal{F}$ and $\mathrm{span}\{\bG_1\}=\mathcal{G}$, then
	\begin{equation}
		\mathrm{dist}(\mathcal{F}, \mathcal{G}) = \|\bF_{2}^{\top}\bN^{-1}\bG_1\|_{2} .
	\end{equation}
\end{lemma}
\begin{proof}
	Several matrix-form expressions of $\Theta(\mathcal{F}, \mathcal{G})$ have been studied in \cite[Section 4]{Knyazev2002}. Here we exploit \cite[Thereom 4.2]{Knyazev2002}, which states that the canonical angles between $\mathcal{F}$ and $\mathcal{G}$ under $\bN^{-1}$-inner product coincide with the canonical angles between $\bL_{N}\mathcal{F}$ and $\bL_{N}\mathcal{G}$ under 2-inner product Let $\bar{\Theta}(\mathcal{F}, \mathcal{G})$ be the maximum canonical angle between $\bL_{N}\mathcal{F}$ and $\bL_{N}\mathcal{G}$ under 2-inner product. By \cite[Theorem 2.2]{Stewart2001matrix} we have
	\begin{align*}
		\mathrm{dist}(\mathcal{F}, \mathcal{G}) =
		\sin\bar{\Theta}(\bL_{N}\mathcal{F}, \bL_{N}\mathcal{G}) = \|(\bL_{N}\bF_2)^{\top}(\bL_{N}\bG_1)\|_2 = \|\bF_{2}^{\top}\bN^{-1}\bG_1\|_{2} .
	\end{align*}
	The first ``=" uses the fact that $(\bL_{N}\bF_1, \bL_{N}\bF_2)$ and $(\bL_{N}\bG_1, \bL_{N}\bG_2)$ are two orthogonal matrices, and $\mathrm{span}\{\bL_{N}\bF_1\}=\bL_{N}\mathcal{F}$ and $\mathrm{span}\{\bL_{N}\bG_1\}=\bL_{N}\mathcal{G}$. \qed
\end{proof}

Now we establish an explicit expression of the distance between $\mathrm{span}\{\bV_k\}$ and $\mathrm{span}\{\bZ_k\}$. To this end, we exploit the following commonly used DPC model for \eqref{gen_regu}:
\begin{equation}\label{dpc}
	|\bu_{A,i}^{\top}\bL_{M}\bb_{\text{true}}| = \gamma_{i}^{1+\beta}, \ \ i=1,\dots,r ,
\end{equation}
where $\beta>0$ controls the decay rate of the Fourier coefficients $\bu_{A,i}^{\top}\bL_{M}\bb_{\text{true}}$ with respect to $\gamma_i$ \cite[\S 4.6]{Hansen2010}. Note that $\bL_{M}\bb=\bL_{M}\bb_{\text{true}}+\bL_{M}\bepsilon$ and $\bL_{M}\bepsilon$ is a Gaussian white noise. Therefore, it follows that $\bu_{A,i}^{\top}\bL_{M}\bb=\bu_{A,i}^{\top}\bL_{M}\bb_{\text{true}}+\bu_{A,i}^{\top}\bL_{M}\bepsilon$, where $\tilde{\bepsilon}=\bU_{A,r}^{\top}\bL_{M}\bepsilon$ is a Gaussian white noise. By DPC,   $\bu_{A,i}^{\top}\bL_{M}\bb_{\text{true}}$ dominates in $\bu_{A,i}^{\top}\bL_{M}\bb$ for small $i$ while $\bu_{A,i}^{\top}\bL_{M}\bepsilon$ dominates for big $i$, thereby it is reasonable to assume that $\bu_{A,i}^{\top}\bL_{M}\bb$ is nonzero for all $1\leq i\leq r$. Based on the above model, we have the follwing result.

\begin{theorem}\label{thm4.1}
	Suppose $\{\gamma_i\}_{i=1}^{r}$ have distinct values. For any $1\leq k \leq r$, let 
	\begin{equation*}
		\bD = \mathrm{diag}(\bD_{A}\bU_{A,r}^{\top}\bL_{M}\bb)
		= \bordermatrix*[()]{%
			\bD_{1}  &  &  k \cr
			&  \bD_{2} & r-k \cr
			k & r-k}, 
    \end{equation*}
    \begin{equation*}         
		\bH_k = 
		\begin{pmatrix}
			1 & \sigma_{1}^{2} & \cdots & \sigma_{1}^{2k-2} \\
			1 & \sigma_{2}^{2} & \cdots & \sigma_{2}^{2k-2} \\
			\vdots & \vdots    & \ddots & \vdots \\
			1 & \sigma_{r}^{2} & \cdots & \sigma_{r}^{2k-2} 
		\end{pmatrix}
		= \bordermatrix*[()]{%
			\bH_{k1} & \ \ k \cr
			\bH_{k2} & \ \ r-k \cr
			k} ,
	\end{equation*}
	and 
	\begin{equation}\label{delta_k}
		\bDelta_k = \bD_{2}\bH_{k2}\bH_{k1}^{-1}\bD_{1}^{-1} \in \mathbb{R}^{(r-k)\times k} .
	\end{equation}
	Then the distance between the $k$-th solution subspace $\mathcal{S}_k=\mathrm{span}\{\bV_k\}$ and $\mathcal{Z}_k=\mathrm{span}\{\bZ_k\}$ is
	\begin{equation}\label{dist_VZ}
		\mathrm{dist}(\mathcal{S}_k, \mathcal{Z}_k) = \frac{\|\bDelta_k\|_2}{\left(1+\|\bDelta_k\|_{2}^{2}\right)^{1/2}}.
	\end{equation}
\end{theorem}
\begin{proof}
	First notice that the diagonals of $\bD_1$ are nonzero and the Vandermonde matrix $\bH_{k1}$ is invertible since $\{\sigma_i\}_{i=1}^{r}=\{\gamma_i\}_{i=1}^{r}$ have distinct values. Lemma \ref{Struct_Vk} implies $\mathcal{S}_k=\mathrm{span}\{\bZ_{r}\bD\bH_{k}\}$. Since
	\begin{align*}
		\bZ_{r}\bD\bH_{k} = \bZ_{r}\begin{pmatrix}
			\bD_{1}\bH_{k1} \\ \bD_{2}\bH_{k2}
		\end{pmatrix} = \bZ_{r}\begin{pmatrix}
			\bI \\ \bDelta_{k}
		\end{pmatrix}\bD_{1}\bH_{k1} ,
	\end{align*} 
	it follows that $\mathcal{S}_k=\mathrm{span}\left\{\bZ_{r}\begin{pmatrix}
		\bI \\ \bDelta_{k}
	\end{pmatrix}\right\}$. Using the relation
	\[ \left(\bZ_{r}\begin{pmatrix}
		\bI \\ \bDelta_{k}
	\end{pmatrix}\right)^{\top}\bN^{-1}\bZ_{r}\begin{pmatrix}
		\bI \\ \bDelta_{k}
	\end{pmatrix} = \bI + \bDelta_{k}^{\top}\bDelta_{k},\]
	one can check that the columns of $\bZ_{r}\begin{pmatrix}
		\bI \\ \bDelta_{k}
	\end{pmatrix}(\bI + \bDelta_{k}^{\top}\bDelta_{k})^{-1/2}$ constitute an $\bN^{-1}$-orthonormal basis of $\mathcal{S}_k$. Partition $\bZ$ as $\bZ=\bordermatrix*[()]{%
		\bZ_{k} & \bZ_{k1} & \bZ_{k2} & \cr
		k & r-k & n-r}$. By Lemma \ref{angle_expres}, we get 
	\begin{align*}
		\mathrm{dist}(\mathcal{S}_k, \mathcal{Z}_k)
		&= \left\|(\bZ_{k1}, \bZ_{k2})^{\top}\bN^{-1}(\bZ_{k}, \bZ_{k1})\begin{pmatrix}
			\bI \\ \bDelta_{k}
		\end{pmatrix}(\bI + \bDelta_{k}^{\top}\bDelta_{k})^{-1/2}\right\|_2 \\
		&= \left\|\begin{pmatrix}
			\bDelta_{k} \\ \mathbf{0}
		\end{pmatrix}(\bI + \bDelta_{k}^{\top}\bDelta_{k})^{-1/2}\right\|_2 
		= \frac{\|\bDelta_k\|_2}{\left(1+\|\bDelta_k\|_{2}^{2}\right)^{1/2}} .
	\end{align*}
	The proof is completed. \qed
\end{proof}

Theorem \ref{thm4.1} generalizes the results in \cite[\S 6.4]{Hansen1998} under a similar model and assumption. The value of $\|\bDelta_k\|_2$ is essential for measuring the distance between $\mathcal{S}_k$ and $\mathcal{Z}_k$, where a qualitative investigation is presented in \cite[\S 6.4]{Hansen1998}. Here we gain more insight into the variation of $\|\bDelta_k\|_2$ as $k$ increases. First, we have
\begin{align}\label{fnorm}
	\|\bDelta_k\|_2 &\leq
	\left\|\bD_2\right\|_2\left\|\bH_{k2}\bH_{k1}^{-1}\right\|_2\left\|\bD_1^{-1}\right\|_2  \nonumber \\
	&=\frac{\sigma_{k+1}}{\sigma_k}\frac{\max_{k+1\leq i\leq r}|\bu_{A,i}^{\top}\bL_{M}\bb|}
	{\min_{1\leq i\leq k}|\bu_{A,i}^{\top}\bL_{M}\bb|}\left\|\bH_{k2}\bH_{k1}^{-1}\right\|_2. 
\end{align}
Following the proof of \cite[Theorem 6.4.1]{Hansen1998}, the $i$-th column of $\bH_{k2}\bH_{k1}^{-1}$ is $\left(L_i^{(k)}(\sigma_{k+1}^2),\ldots,L_i^{(k)}(\sigma_{r}^2)\right)^{\top}$, where 
$L_i^{(k)}(\lambda) = \prod\limits_{j=1,j\neq i}^k(\lambda-\sigma_j^2)/(\sigma_i^2-\sigma_j^2)$ for $1\leq i\leq k$, $k>1$, and $L_1^{(1)}(\lambda)\equiv 1$. Notice that $\left|L_i^{(k)}(\lambda)\right|$ is monotonically decreasing for $\lambda\in[0, \sigma_{k}^{2}]$. The upper bound on it is $\left|L_i^{(k)}(0)\right|$. Let $\left|L_{i_0}^{(k)}(0)\right|=\max_{1\leq i \leq k}\left|L_i^{(k)}(0)\right|$. Then we obtain
\begin{align}\label{bnd_H}
	\left\|\bH_{k2}\bH_{k1}^{-1}\right\|_2 
	&\leq \sqrt{k}\max_{1\leq i\leq k}\left\|\left[\bH_{k2}\bH_{k1}^{-1}\right]_i\right\|_2  \nonumber \\
	&\leq \sqrt{k}\max_{1\leq i\leq k}\sqrt{r-k}\left|L_i^{(k)}(0)\right| = \sqrt{k(r-k)}\left|L_{i_0}^{(k)}(0)\right| ,
\end{align}
where $[\cdot]_i$ denotes the $i$-th column of a matrix. 

To further investigate the upper bound on $\|\bDelta_k\|_2$ in \eqref{fnorm}, we use the simplified model $\sigma_i=\gamma_i=\gamma_{0}e^{-\alpha i}$ for $i=1,\dots, r$ to describe the decaying rate of $\gamma_i$, i.e. they decay exponentially \footnote{This model describes the so called ``severely ill-posed'' forward operator $\bA$. Another model $\gamma_i=\gamma_{1}i^{-\alpha}$ for $i=1,\dots, r$ describes the ``modestly/mildly ill-posed'' $\bA$; see \cite[\S 2.2]{Engl2000} for more discussions and examples.}. For small $k$ that the noise components in $\bu_{A,i}^{\top}\bL_{M}\bb$ do not dominate, using the model \eqref{dpc} we have $\frac{\sigma_{k+1}}{\sigma_k}\frac{\max_{k+1\leq i\leq r}|\bu_{A,i}^{\top}\bL_{M}\bb|}{\min_{1\leq i\leq k}|\bu_{A,i}^{\top}\bL_{M}\bb|}\approx (\gamma_{k+1}/\gamma_k)^{2+\beta}=e^{-\alpha(2+\beta)}$. Then we estimate the bound on $\left\|\bH_{k2}\bH_{k1}^{-1}\right\|_2$ in \eqref{bnd_H}. Obviously $\left|L_{i_0}^{(1)}(0)\right|=1$. For $k\geq 2$, we have
\begin{align*}
	\left|L_{i_0}^{(k)}(0)\right|
	&= \prod\limits_{j=1,j\neq i_0}^k
	\left|\frac{\sigma_j^2}{\sigma_j^2-\sigma_{i_0}^2}\right|
	= \prod\limits_{j=1}^{i_0-1}\frac{\sigma_j^2}{\sigma_j^2-\sigma_{i_0}^2}
	\prod\limits_{j=i_0+1}^{k}\frac{\sigma_j^2}{\sigma_{i_0}^2-\sigma_{j}^2} \\
	& =\prod\limits_{j=1}^{i_0-1}\frac{1}
	{1-e^{-2(i_0-j)\alpha}}
	\prod\limits_{j=i_0+1}^{k}\frac{1}
	{e^{2(j-i_0)\alpha}-1} \\
	& =\prod\limits_{j=1}^{i_0-1}\frac{1}
	{1-e^{-2j\alpha}}
	\prod\limits_{j=1}^{k-i_0}\frac{1}{1-e^{-2j\alpha}}\prod\limits_{j=1}^{k-i_0}\frac{1}{e^{2j\alpha}} \\
	&\leq \prod\limits_{j=1}^{i_0-1} \left(1+\frac{1}{2j\alpha}\right) \prod\limits_{j=1}^{k-i_0} \left(1+\frac{1}{2j\alpha}\right) e^{-\alpha(k-i_0)(k-i_0+1)},
\end{align*}
where we use $1/(1-e^{-x})\leq 1+1/x$ for $x>-1$. Notice that for $i_0=1$ the first term vanishes, and for $i_0=k$ the second and third terms vanish. Define $S_1=\prod\limits_{j=1}^{i_0-1} \left(1+\frac{1}{2j\alpha}\right)$. Then
\begin{align*}
	\ln S_1 = \sum_{j=1}^{i_0-1}\ln\left( 1+\frac{1}{2j\alpha}\right)
	\leq \frac{1}{2\alpha}\sum_{j=1}^{i_0-1}\frac{1}{j}
	\leq \frac{1}{2\alpha}\left(1 + \int_{1}^{i_0-1}\frac{1}{j} dx\right)
	= \frac{1+\ln(i_0-1)}{2\alpha} .
\end{align*}
Similarly, we have $\ln \prod\limits_{j=1}^{k-i_0} \left(1+\frac{1}{2j\alpha}\right) \leq \frac{1+\ln(k-i_0)}{2\alpha}$. Therefore, we obtain
\begin{align*}
	\ln \left|L_{i_0}^{(k)}(0)\right| 
	&\leq \frac{1}{2\alpha}\left(2+\ln(i_0-1)(k-i_0)\right) -\alpha(k-i_0)(k-i_0+1) \\
	&\leq \frac{1}{2\alpha}\left(2+\ln\frac{(k-1)^2}{4}\right)-\alpha(k-i_0)(k-i_0+1),
\end{align*}
which leads to
\begin{equation*}
	\left|L_{i_0}^{(k)}(0)\right| \leq \left(\frac{e(k-1)}{2} \right)^{\frac{1}{\alpha}} e^{-\alpha(k-i_0)(k-i_0+1)} .
\end{equation*}
If $i_0$ keeps around a constant value when $k$ increases, the above bound is a gradually decreasing quantity. If $k-i_0$ always remains a constant or decreases when $k$ increases, the above bound is pessimistic since it increases at the speed of $k^{1/\alpha}$. In this case, we derive an asymptotic bound:
\begin{align*}
	\left|L_{i_0}^{(k)}(0)\right|
	&\leq \prod\limits_{j=1}^{i_0-1}\frac{1}{1-e^{-2j\alpha}}
	= \prod\limits_{j=1}^{i_0-1}\left(1+\mathcal{O}(e^{-2j\alpha})\right) \\
	&= 1+ \mathcal{O}\left(\sum_{j=1}^{i_0-1}e^{-2j\alpha}\right)
	= 1+ \mathcal{O}\left(\frac{e^{-2\alpha}}{1-e^{-2\alpha}}\right) .
\end{align*}
A sharp upper bound on $\left|L_{i_0}^{(k)}(0)\right|$ is extremely difficult since $i_0$ may vary for different $k$, but in the worst case it grows from 1 very slightly controlled by $\mathcal{O}(e^{-2\alpha})$. 

From the above derivation, we have
\begin{equation}
	\|\bDelta_k\|_2 \lesssim \sqrt{k(r-k)}e^{-\alpha(2+\beta)}\left|L_{i_0}^{(k)}(0)\right| 
\end{equation}
for small $k$ that the noise in $|\bu_{A,k}^{\top}\bL_{M}\bb|$ is negligible. The bounded factor $\sqrt{k(r-k)}$ is an overestimate that can be reduced more or less, and $\left|L_{i_0}^{(k)}(0)\right|$ is slightly larger than 1 in the worst case. Therefore, we can loosely speak that $\|\bDelta_k\|_2$ remains small before the noise dominates, which implies that $\mathcal{S}_k$ captures main information about $\bx$ contained in $\mathcal{Z}_k$. Going further, we have the following result.

\begin{theorem}\label{lanc_process}
	Let the compact singular value decomposition (SVD) of $\bB_k$ be $\bB_k=\bJ_{k}\bGamma_{k}\bP_{k}^{\top}$, where $\bGamma_k=\mathrm{diag}(\theta_{1}^{(k)},\dots,\theta_{k}^{(k)})$ with $\theta_{1}^{(k)}>\theta_{2}^{(k)}>\cdots>\theta_{k}^{(k)}>0$ and $\bP_{k}=(\bp_{1}^{(k)},\dots,\bp_{k}^{(k)})$. Then each pair $\left((\theta_{i}^{(k)})^2, \bV_{k}\bp_{i}^{(k)}\right)$ converges to one of the generalized eigenvalue-vector pairs of $\bA^{\top}\bM^{-1}\bA\bz=\xi\bN^{-1}\bz$ as $k$ increases.
\end{theorem}
\begin{proof}
	By \eqref{GKB23} and \eqref{GKB33}, we have
	\begin{align*}
		\bA^{\top}\bM^{-1}\bA\bV_k
		= (\bA^{\top}\bM^{-1}\bU_{k+1})\bB_{k}
		= \bN^{-1}\bV_{k}\bB_{k}^{\top}\bB_{k} + \alpha_{k+1}\beta_{k+1}\bN^{-1}\bv_{k+1}\be_{k}^{\top}, 
	\end{align*}
	where $\be_{k}$ is the $k$-th column of the identity matrix of order $k$. By \eqref{GKB2} and \eqref{GKB13}, we have
	\begin{equation*}
		\alpha_{1}\beta_{1}\bv_{1} = \bN\bA^{\top}\bM^{-1}\beta_{1}\bu_{1} = \bN\bA^{\top}\bM^{-1}\bb .
	\end{equation*}
	Notice that $\bB_{k}^{\top}\bB_{k}$ is a symmetric tridiagonal matrix with diagonals $\alpha_{i}^{2}+\beta_{i+1}^{2}$ and subdiagonals $\alpha_{i}\beta_{i}$ for $i=1,\dots,k$. The above two relations indicate that $\bT_k$ is the Ritz-Galerkin projection of $\bA^{\top}\bM^{-1}\bA$ onto subspace $\mathrm{span}\{\bV_k\}$ generated by the Lanczos tridiagonalization of $\bA^{\top}\bM^{-1}\bA$ with starting vector $\bN\bA^{\top}\bM^{-1}\bb$ under $\bN^{-1}$-inner product \cite[\S 5.5]{Bai2000}. Since the eigenvalue decomposition of $\bB_{k}^{\top}\bB_{k}$ is $\bB_{k}^{\top}\bB_{k}=\bP_{k}\bGamma_{k}^{2}\bP_{k}^{\top}$, by the convergence theory of Lanczos tridiagonalization for generalized eigenvalue problems, we obtain the desired result. \qed
\end{proof}

By \eqref{gen_eig}, the generalized eigenvalue-vector pairs for $\bA^{\top}\bM^{-1}\bA\bz=\xi\bN^{-1}\bz$ are $\{(\gamma_{i}^{2}, \bz_{i})\}_{i=1}^{n}$. The convergence behavior of $\left((\theta_{i}^{(k)})^2, \bV_{k}\bp_{i}^{(k)}\right)$ as $k$ increases is governed by the Kaniel-Paige-Saad theory that is similar to the eigenvalue problem of a single matrix \cite[\S 10.1.5]{Golub2013}, which states that it has a faster convergence to those generalized eigenvalues (and corresponding eigenvectors) lying in the two ends of the spectrum and well separated. Since $\gamma_{i}^{2}$ decreases and clusters at zero, we can hope for a faster convergence of $\left((\theta_{i}^{(k)})^2, \bV_{k}\bp_{i}^{(k)}\right)$ to some largest generalized eigenvalue and corresponding vector. Therefore, the information in $\mathcal{S}_k$ can be extracted to approximate dominant generalized eigenvectors $\bz_i$. These vectors constitute basic features about $\bx$ encoded by the Gaussian prior. In fact, the regularized solution computed by genGKB\_SPR contains those dominant $\bz_i$ while filtering out others. To reveal it, we need the following lemma, which generalizes \cite[Property 2.8]{Van1986rate} for the conjugate gradient (CG) method to a rank-deficient matrix.

\begin{lemma}\label{lem4.3}
	For a symmetric positive semi-definite linear system $\bC\bff=\bd$ with $\bd\in\mathcal{R}(\bC)$, where $\mathcal{R}(\cdot)$ is the range space of a matrix, the $k$-th CG solution of it with $\bd_0=\mathbf{0}$ can be written as $\bff_k=\left(\bI-\mathcal{R}_{k}(\bC)\right)\bC^{\dag}\bd$, where $\mathcal{R}_{k}$ is the normalized characteristic polynomial of $\bT_k$ as \eqref{Tk} such that $\mathcal{R}_{k}(0)=1$.
\end{lemma}
\begin{proof}
	Suppose the $k$-step standard Lanczos tridiagonalization of $\bC$ with starting vector $\bd$ is written as the following matrix-form:
	\begin{align}
		\bC\bQ_{k} &= \bQ_{k}\bT_{k} + \delta_{k+1}\bq_{k+1}\be_{k}^{\top} , \label{lanc1} \\
		\delta_{1}\bq_1 &= \bd ,  \label{lanc2}
	\end{align}
	where $\{\bq_i\}_{i=1}^{k+1}$ are orthonormal and 
	\begin{equation}\label{Tk}
		\bT_{k} = \begin{pmatrix}
			\eta_1 & \delta_2 & & \\
			\delta_2 & \eta_2 & \ddots & \\
			& \ddots & \ddots & \delta_k \\
			&        & \delta_k & \eta_k
		\end{pmatrix} .
	\end{equation}
	We note that this process must terminate, i.e. $\delta_{k+1}=0$ for at most $k=l$ steps, where $l=\mathrm{rank}(\bC)$, and the corresponding solution is just $\bC^{\dag}\bd$; if the process does not terminate at the $k$-th step, then $\bT_k$ is invertible \cite[\S 4.2.3--\S 4.2.5]{bjorck2015numerical}. To prove this lemma, we use the fact that the $k$-th CG solution of $\bC\bff=\bd$ with $\bd_0=\mathbf{0}$ is $\bff_k=\bQ_{k}\bT_{k}^{-1}\delta_1\be_1$.
	
	Let the compact SVD of $\bC$ be $\bC=\bV_{C}\bSigma_{C}\bV_{C}^{\top}$, where $\bSigma_{C}\in\mathbb{R}^{l\times l}$ has positive diagonals and $\bV_{C}$ has $l$ orthonormal columns. Since $\bd\in\mathcal{R}(\bC)=\mathcal{R}(\bV_{C})$, it follows that $\bq_i\in\mathcal{R}(\bV_{C})$ for $i=1,\dots,k+1$. Let $\bq_{i}=\bV_{C}\bar{\bq}_{i}$ with $\bar{\bq}_{i}\in\mathbb{R}^{l}$ and $\bar{\bQ}_{k}=(\bar{\bq}_{1},\dots,\bar{\bq}_{k})$. It follows that $\{\bar{\bq}_i\}_{i=1}^{k+1}$ are orthonormal. From \eqref{lanc1} and \eqref{lanc2} we get
	\begin{align*}
		\bSigma_{C}\bar{\bQ}_{k} &= \bar{\bQ}_{k}\bT_{k}+\delta_{k+1}\bar{\bq}_{k+1}\be_{k}^{\top},   \\
		\delta_{1}\bar{\bq}_{1}  &= \bV_{C}^{\top}\bd ,
	\end{align*}
	which is just the matrix-form of the Lanczos tridiagonalization of $\bSigma_{C}$ with starting vector $\bV_{C}^{\top}\bd$. Let $\bar{\bff}_k=\bar{\bQ}_{k}\bT_{k}^{-1}\delta_1\be_1$. Since $\bSigma_{C}$ is positive definite, by \cite[Property 2.8]{Van1986rate} we have $\bar{\bff}_k=\left(\bI-\mathcal{R}_{k}(\bSigma_{C})\right)\bSigma_{C}^{-1}\bV_{C}^{\top}\bd$. Using $\bQ_{k}=\bV_{C}\bar{\bQ}_{k}$, we obtain
	\begin{align*}
		\bff_k 
		&= \bV_{C}\bar{\bff}_k= \bV_{C}\left(\bI-\mathcal{R}_{k}(\bSigma_{C})\right)\bSigma_{C}^{-1}\bV_{C}^{\top}\bd \\
		&= \left(\bI-\mathcal{R}_{k}(\bV_{C}\bSigma_{C}\bV_{C}^{\top})\right)\bV_{C}\bSigma_{C}^{-1}\bV_{C}^{\top}\bd \\
		&= \left(\bI-\mathcal{R}_{k}(\bC)\right)\bC^{\dag}\bd .
	\end{align*}
	The proof is completed. \qed
\end{proof}

\begin{theorem}\label{filter_express}
	The $k$-th regularized solution computed by genGKB\_SPR has the expression
	\begin{equation}\label{filter_solution}
		\bx_k = \sum_{i=1}^{r}f_{i}^{(k)}\frac{\bu_{A,i}^{\top}\bL_{M}\bb}{\sigma_i}\bz_i
	\end{equation}
	with filter factors
	\begin{equation}\label{filter1}
		f_{i}^{(k)} = 1 - \prod_{j=1}^{k}\frac{(\theta_{j}^{(k)})^2-\gamma_{i}^{2}}{(\theta_{j}^{(k)})^2}, \ \ i=1,\dots,r.
	\end{equation}
\end{theorem}
\begin{proof}
	Following the proof of Theorem \ref{lanc_process}, we have
	\begin{align*}
		(\bL_{N}^{-\top}\bA^{\top}\bM^{-1}\bA\bL_{N}^{-1})\bL_{N}\bV_k &= \bL_{N}\bV_{k}(\bB_{k}^{\top}\bB_{k}) + \alpha_{k}\beta_{k}\bL_{N}\bv_{k+1}\be_{k}, \\
		\alpha_{1}\beta_{1}\bL_{N}\bv_{1} &= \bL_{N}^{-\top}\bA^{\top}\bM\bb ,
	\end{align*}
	and $\{\bL_{N}\bv_i\}_{i=1}^{k+1}$ are 2-orthonormal. This implies that $\bL_{N}\bv_i$ are the Lanczos vectors generated by the standard Lanczos tridiagonalization of $\bL_{N}^{-\top}\bA^{\top}\bM^{-1}\bA\bL_{N}^{-1}$ with starting vector $\bL_{N}^{-\top}\bA^{\top}\bM\bb$. Let $\tilde{\bx}_{k}=\bL_{N}\bx_k$. Then by \eqref{k_dim_ls} we have $\tilde{\bx}_k=\bL_{N}\bV_{k}(\bB_{k}^{\top}\bB_{k})^{-1}\bB_{k}^{\top}\beta_{1}\be_1=(\bL_{N}\bV_{k})\bT_{k}^{-1}\alpha_{1}\beta_{1}\be_1$. Therefore, $\tilde{\bx}_k$ is the $k$-th CG solution of the symmetric positive semi-definite linear system
	\begin{equation}\label{nm_eq}
		\bL_{N}^{-\top}\bA^{\top}\bM^{-1}\bA\bL_{N}^{-1}\bx = \bL_{N}^{-\top}\bA^{\top}\bM\bb .
	\end{equation}
	Note that 
	\[\bL_{N}^{-\top}\bA^{\top}\bM\bb=(\bL_{M}\bA\bL_{N}^{-1})^{\top}\bL_{M}b, \ \ \ 
	\bL_{N}^{-\top}\bA^{\top}\bM^{-1}\bA\bL_{N}^{-1}=(\bL_{M}\bA\bL_{N}^{-1})^{\top}(\bL_{M}\bA\bL_{N}^{-1}) . \]
	It follows that $\bL_{N}^{-\top}\bA^{\top}\bM\bb\in\mathcal{R}(\bL_{N}^{-\top}\bA^{\top}\bM^{-1}\bA\bL_{N}^{-1})$ and 
	\begin{align*}
		& \ \ \ \ (\bL_{N}^{-\top}\bA^{\top}\bM^{-1}\bA\bL_{N}^{-1})^{\dag}\bL_{N}^{-\top}\bA^{\top}\bM\bb \\
		&= (\bL_{M}\bA\bL_{N}^{-1})^{\dag}\bL_{M}b = \left(\bU_{A}\bSigma_{A}(\bL_{N}\bZ)^{-1}\right)^{\dag}\bL_{M}b \\
		&= \bL_{N}\bZ\bSigma_{A}^{\dag}\bU_{A}^{\top}\bL_{M}b ,
	\end{align*}
	where we use \eqref{gsvd} and the property that $\bL_{N}\bZ$ has orthonormal columns. Therefore, by Lemma \ref{lem4.3} we have 
	\begin{equation}\label{filter2}
		\tilde{\bx}_k = \left(\bI-\mathcal{R}_k(\bL_{N}^{-\top}\bA^{\top}\bM^{-1}\bA\bL_{N}^{-1})\right)\bL_{N}\bZ\bSigma_{A}^{\dag}\bU_{A}^{\top}\bL_{M}b ,
	\end{equation}
	where $\mathcal{R}_k$ is the normalized characteristic polynomial of $\bB_{k}^{\top}\bB_{k}$ such that $\mathcal{R}_{k}(0)=1$.
	
	By \eqref{gen_eig}, we have $\bL_{N}^{-\top}\bA^{\top}\bM^{-1}\bA\bL_{N}^{-1}=\bL_{N}^{-\top}\bZ^{-\top}\bSigma_{A}^{\top}\bSigma_{A}\bZ^{-1}\bL_{N}^{-1}$. Since $\bZ^{-1}\bL_{N}^{-1}\bL_{N}^{-\top}\bZ^{-\top}=\left(\bZ^{\top}\bN^{-1}\bZ\right)^{-1}=\bI$, we get $\bI-\mathcal{R}_k(\bL_{N}^{-\top}\bA^{\top}\bM^{-1}\bA\bL_{N}^{-1})=\bL_{N}^{-\top}\bZ^{-\top}\bLambda\bZ^{-1}\bL_{N}^{-1}$, where $\bLambda=\mathrm{diag}\left(\{1-\mathcal{R}_k(\sigma_{i}^{2})\}_{i=1}^{n}\right)$, where $\sigma_i:=0$ for $r+1\leq i \leq n$.
	Therefore, \eqref{filter2} leads to $\tilde{\bx}_k=\bL_{N}^{-\top}\bZ^{-\top}\bLambda\bSigma_{A}^{\dag}\bU_{A}^{\top}\bL_{M}\bb$, and 
	\begin{align*}
		\bx_{k} 
		&= \bL_{N}^{-1}\bL_{N}^{-\top}\bZ^{-\top}\bLambda\bSigma_{A}^{\dag}\bU_{A}^{\top}\bL_{M}\bb 
		= \bZ\bLambda\bSigma_{A}^{\dag}\bU_{A}^{\top}\bL_{M}\bb  \\
		&= \begin{pmatrix}
			\bZ_{r} & \bZ_{r1}
		\end{pmatrix}
		\begin{pmatrix}
			\bLambda_{r} & \\
			& \bLambda_{r1}
		\end{pmatrix}
		\begin{pmatrix}
			\bD_{A}^{-1} &  \\
			& \mathbf{0}
		\end{pmatrix}
		\begin{pmatrix}
			\bU_{A,r}^{\top} \\ \bU_{A,r1}^{\top}
		\end{pmatrix}\bL_{M}\bb  \\
		&= \bZ_{r}\bLambda_{r}\bD_{A}^{-1}\bU_{A,r}^{\top}\bL_{M}\bb ,
	\end{align*}
	where $\bZ$, $\Lambda$ and $\bU_{A}$ are partitioned such that the above matrix multiplication is reasonable. By Theorem \ref{lanc_process}, the eigenvalues of $\bB_{k}^{\top}\bB_{k}$ are $\{(\theta_{j}^{(k)})^2\}_{i=1}^k$. If follows that $\mathcal{R}_k((\theta_{j}^{(k)})^2)=0$ for $1\leq i \leq k$. Since $\mathcal{R}_{k}(0)=1$ and the degree of $\mathcal{R}_{k}$ is $k$, we have
	\begin{equation*}
		\mathcal{R}_k(\theta) = \prod_{j=1}^{k}\frac{(\theta_{j}^{(k)})^2-\theta}{(\theta_{j}^{(k)})^2} .
	\end{equation*}
	Therefore, we finally obtain $\bx_{k}=\sum_{i=1}^{r}f_{i}^{(k)}\frac{\bu_{A,i}^{\top}\bL_{M}\bb}{\sigma_i}\bz_i$ with $f_{i}^{(k)}$ defined as \eqref{filter1}. \qed
\end{proof}

Theorem \ref{filter_express} shows that $\bx_k$ has a filtered GSVD expansion form. If the $k$ Ritz values $\{(\theta_{j}^{(k)})^2\}_{i=1}^{k}$ approximate the first $k$ generalized singular values $\{\gamma_i^2\}_{i=1}^{k}$ in natural order, i.e. $(\theta_{j}^{(k)})^2\approx\gamma_i^2$ for $i=1,\dots,k$, from \eqref{filter1} we can justify that $f_{i}^{(k)}\approx 1$ for $i=1,\dots,k$ and $f_{i}^{(k)}$ decreases monotonically to zero for $i=k+1,\dots,r$. Comparing \eqref{filter_solution} with \eqref{Tikh_gsvd} and \eqref{2.17}, we find that $\bx_k$ contains the first $k$ dominant GSVD components and filters the others. Therefore, as $k$ increases, the iterative solution first approximates $\bx_{\text{true}}$, but as $k$ becomes too big, $\bx_k$ will diverge from $\bx_{\text{true}}$ since too many noisy components are contained. The transition point is the semi-convergence point, precisely what the early stopping rules aim to estimate.

\section{Experimental Results} \label{sec5}
We conduct several numerical experiments to demonstrate the performance of our proposed algorithm. The generalized hybrid iterative method proposed in \cite{Chung2017} is used as a comparison, which is also based on gen-GKB and equivalent to solving \eqref{Bayes1} by projecting it onto the solution space $\mathrm{span}\{\bV_k\}$ to get 
\begin{equation}\label{hyb_k}
	\min_{\bx=\bV_k\by}\{\|\bA\bx-\bb\|_{\bM^{-1}}^{2} + \lambda\|\bx\|_{\bN^{-1}}^2\}
	= \min_{y\in\mathbb{R}^{k}}\|\bB_k\by-\beta_1\be_{1}\|_{2}^2 + \lambda \|\by\|_{2}^2.
\end{equation}
In order to solve the above $k$ dimensional projected problem, at each iteration, the weighted GCV (WGCV) method \cite{Chungnagy2008} is adopted to determine a regularization parameter $\lambda_k$ to form $\lambda_k\|\by\|_{2}^2$. In the experiments, we name the generalized hybrid iterative method with WGCV as genHyb\_WGCV. The convergence behaviors of the two methods are shown by plotting the variation of relative reconstruction error $ \|\bx_k-\bx_{\text{true}}\|_2/\|\bx_{\text{true}}\|_2$ with respect to $k$. The reconstructed solutions are also drawn to further compare the two methods. All the experiments are performed in MATLAB R2019b, where some codes in \cite{Hansen2007,Gazzola2019} are exploited. 

\subsection{Small-scale problems}
We first choose two small-scale linear inverse problems from \cite{Hansen2007} for the test. In this case, we can factorize $\bM^{-1}$ and $\bN^{-1}$ to form \eqref{gen_regu}, and solve it directly to find the optimal Tikhonov regularization parameter $\lambda_{opt}$ and corresponding solution $\bx_{opt}$, that is $\lambda_{opt}=\min_{\lambda>0}\|\bx_{\lambda}-\bx_{\text{true}}\|_{2}$. We use this optimal solution as a baseline for comparing the two methods.

\paragraph{Experiment 1.}
The first small-scale test problem is {\sf gravity}, which is a one-dimensional model for gravity surveying. It aims to reconstruct a mass distribution $f(t)$ located at depth $d$ from a noisy vertical component of the gravity field $g(s)$ measured at the surface. The forward model can be expressed by the first-kind Fredholm integral equation
\begin{equation}\label{fred}
	g(s) = \int_{a}^{b}K(s,t)f(t)\mathrm{d}t
\end{equation}
with $t\in[a, b]=[0, 1]$ and $s\in [0,1]$, where the kernel $K$ and true solution $f$ are given by
\begin{align*}
	K(s, t) &= d \left(d^2+(s-t)^2\right)^{-3/2}, \ \ \ d = 0.25, \\
	f(t) &= \sin(\pi t) + 0.5\sin(2\pi t).
\end{align*}
This equation is discretized by the midpoint quadrature rule with $m=n=2000$ uniform points, leading to the matrix $\bA\in\mathbb{R}^{2000\times 2000}$, $\bx_{\text{true}}\in\mathbb{R}^{2000}$ and $\bb_{\text{true}}=\bA\bx_{\text{true}}$. We construct a discrete Gaussian white noise $\bepsilon\sim\mathcal{N}(\mathbf{0},\sigma^2\bI)$ such that the noisy level $\mathbb{E}[\|\bepsilon\|_{2}]/\|\bb_{\text{true}}\|_{2}\approx \left(\mathbb{E}[\|\bepsilon\|_{2}^2]/\|\bb_{\text{true}}\|_{2}^2\right)^{1/2}=\sqrt{m}\sigma/\|\bb_{\text{true}}\|_{2}=5\times 10^{-3}$, and then let $\bb=\bb_{\text{true}}+\bepsilon$. The true solution and noisy observed data are shown in Figure \ref{fig1}.

\begin{figure}[htbp]
	\centering
	\subfloat 
	{\label{fig:1a}\includegraphics[width=0.34\textwidth]{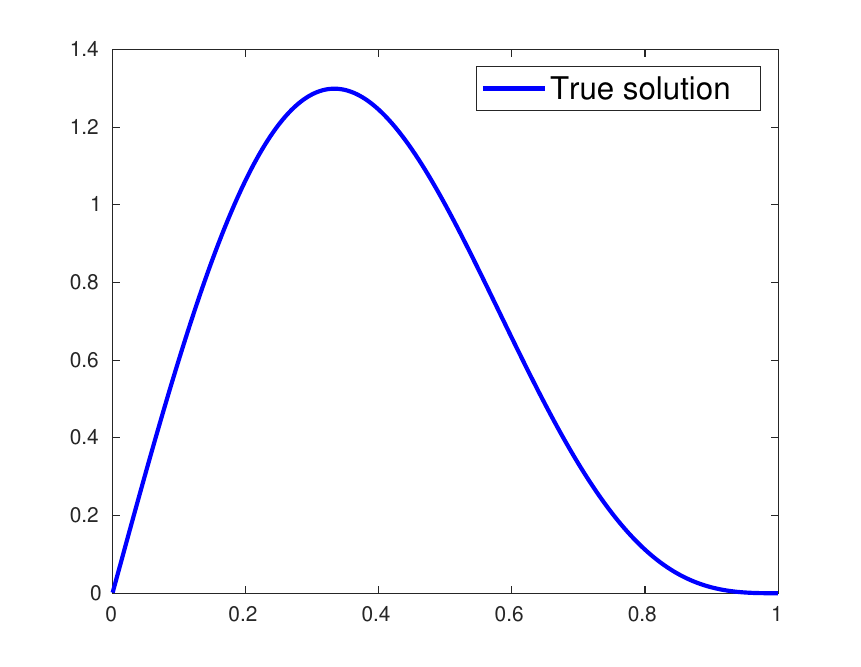}}
	\subfloat
	{\label{fig:1b}\includegraphics[width=0.36\textwidth]{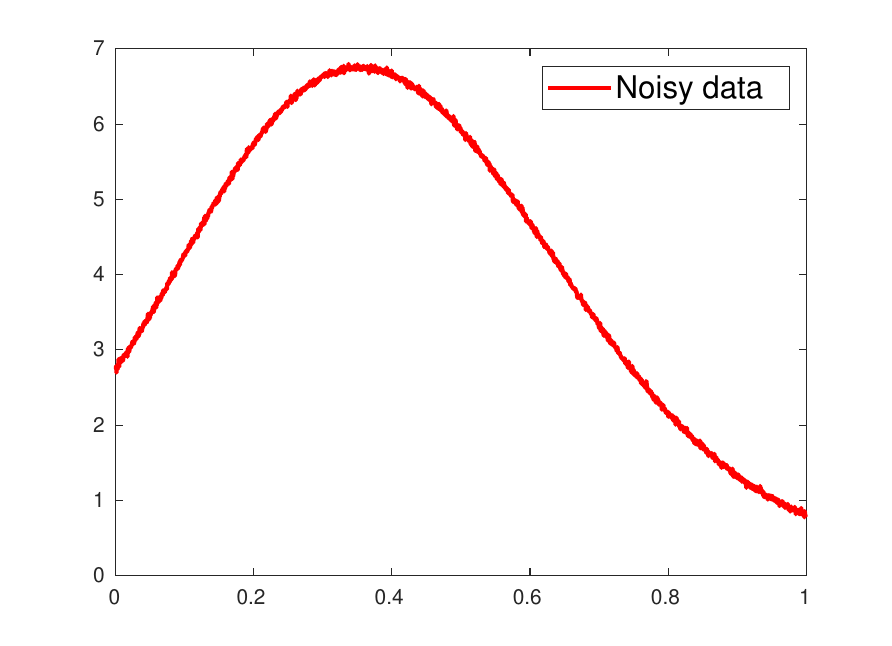}}
	\caption{Illustration of the true solution and noisy observed data for {\sf gravity}.}
	\label{fig1}
\end{figure}

To reconstruct $\bx$ from noisy $\bb$, we assume a Gaussian prior $\bx\sim\mathcal{N}(\boldsymbol{0}, \lambda^{-1}\bN)$ with $\bN$ coming from the Gaussian kernel $\kappa_{G}$, i.e. the $ij$ element of $\bN$ is
\begin{equation*}
	\bN_{ij} = \kappa_{G}(r_{ij}), \ \ \ \kappa_{G}(r):=\exp\left(-r^2/(2l^2)\right) , 
\end{equation*}
where $r_{ij}=\|\bp_i-\bp_j\|_{2}$ and $\{\bp_{i}\}_{i=1}^{n}$ are discretized points of the domain of $f$ (i.e. $f(\bp_i)=[\bx_{\text{true}}]_{i}$). The parameter $l$ is set as $l=0.1$.

\begin{figure}[!htbp]
	\centering
	\subfloat
	{\label{fig:2a}\includegraphics[width=0.42\textwidth]{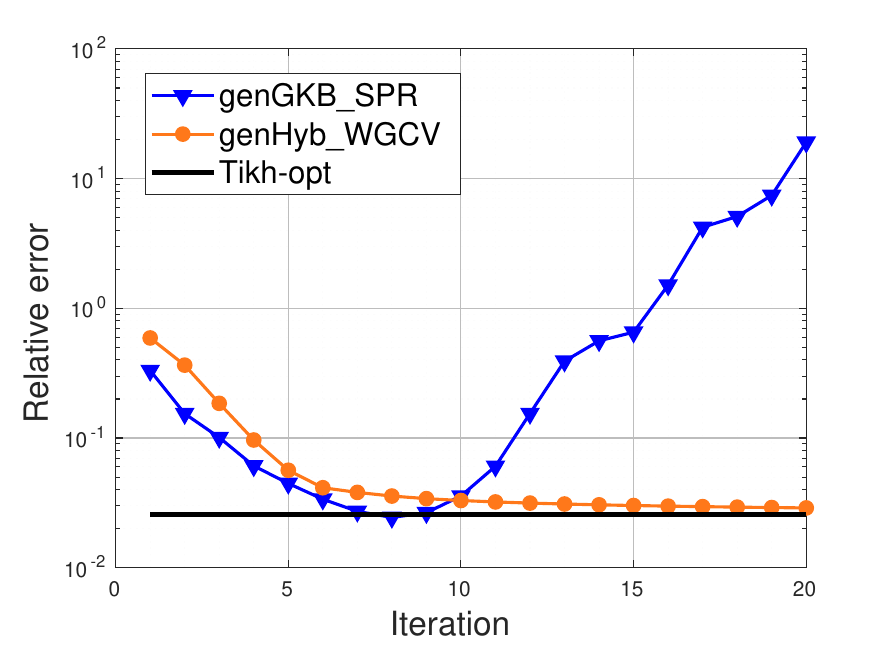}}
	\subfloat
	{\label{fig:2b}\includegraphics[width=0.42\textwidth]{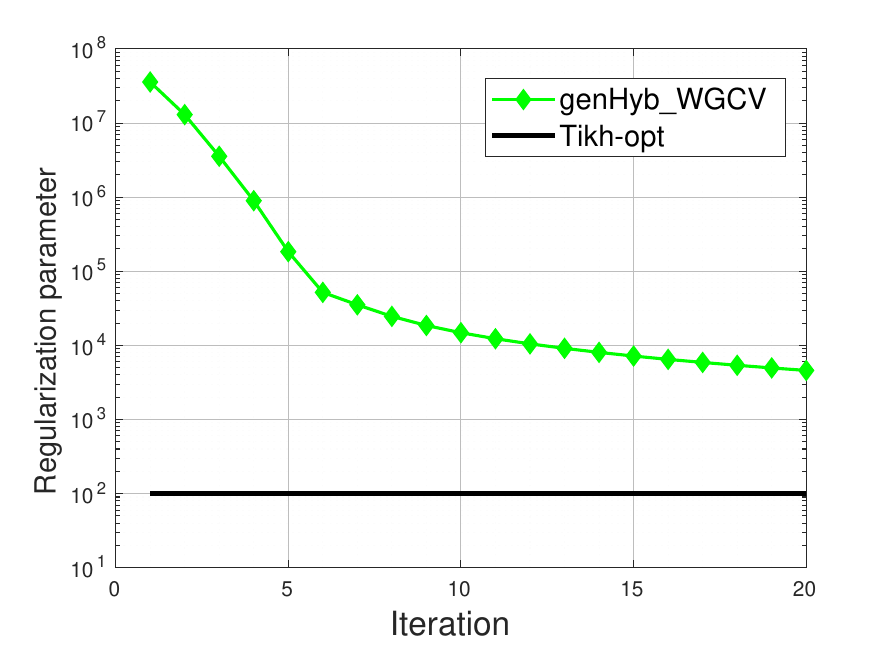}}
	\caption{Comparison of convergence behaviors between genGKB\_SPR and genHyb\_WGCV for {\sf gravity}. Left: relative error curves of regularized solution with respect to iteration number; Right: convergence trend of regularization parameters of genHyb\_WGCV towards $\lambda_{opt}$.}
	\label{fig2}
\end{figure}

\begin{figure}[htbp]
	\centering
	\subfloat 
	{\label{fig:3a}\includegraphics[width=0.35\textwidth]{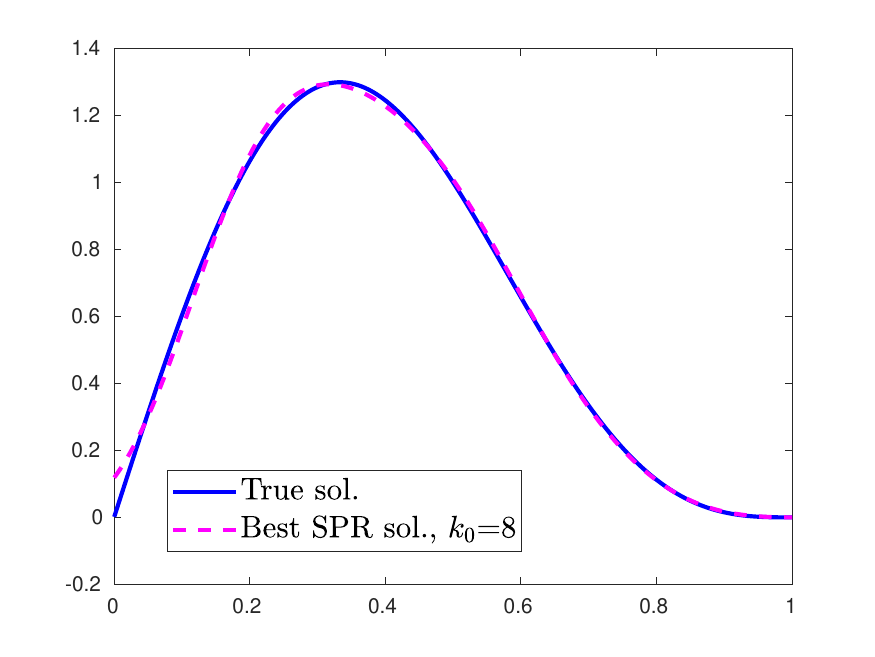}}
	\subfloat{\label{fig:3b}\includegraphics[width=0.36\textwidth]{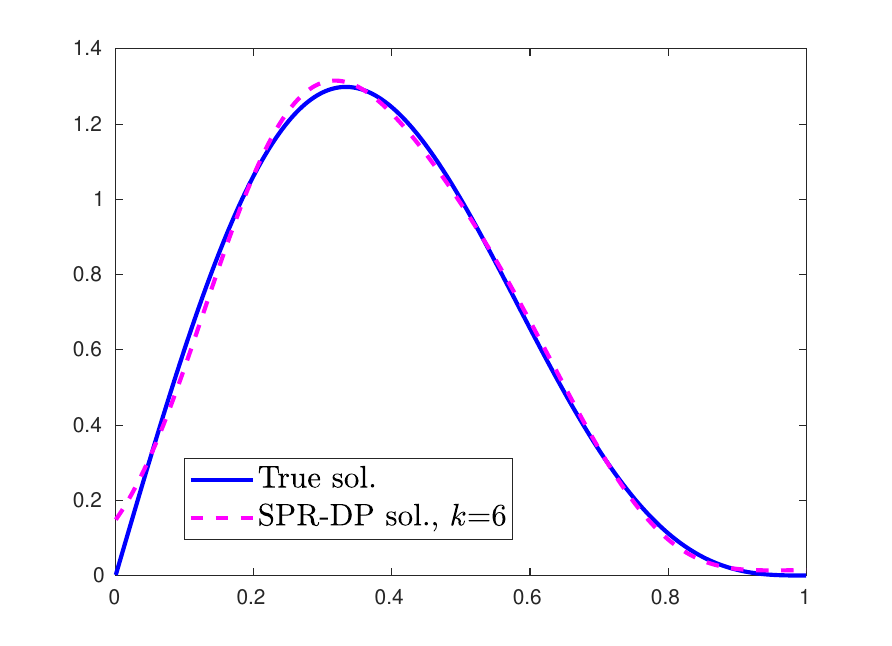}}
	\vspace{-5mm}
	\subfloat 
	{\label{fig:3c}\includegraphics[width=0.34\textwidth]{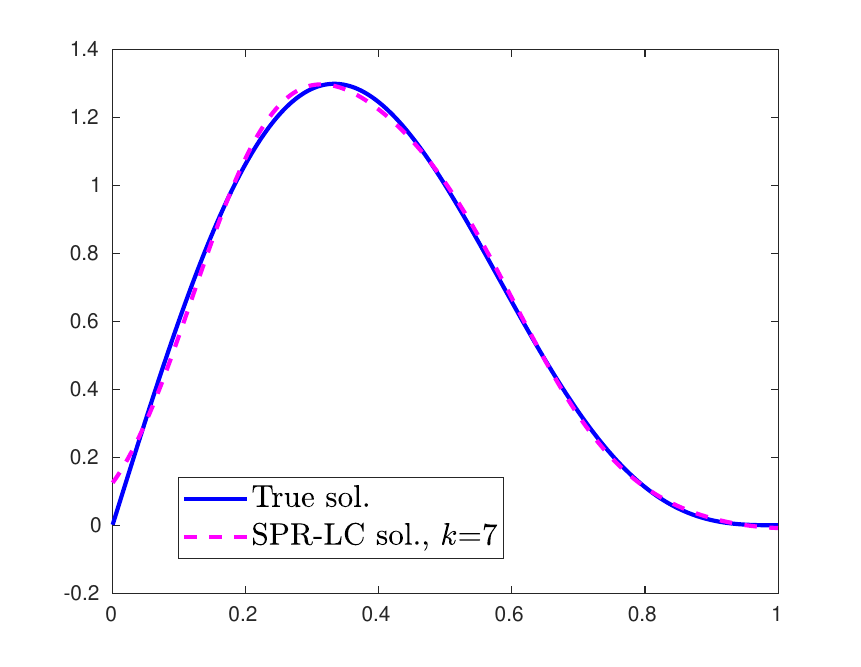}}
	\subfloat{\label{fig:3d}\includegraphics[width=0.36\textwidth]{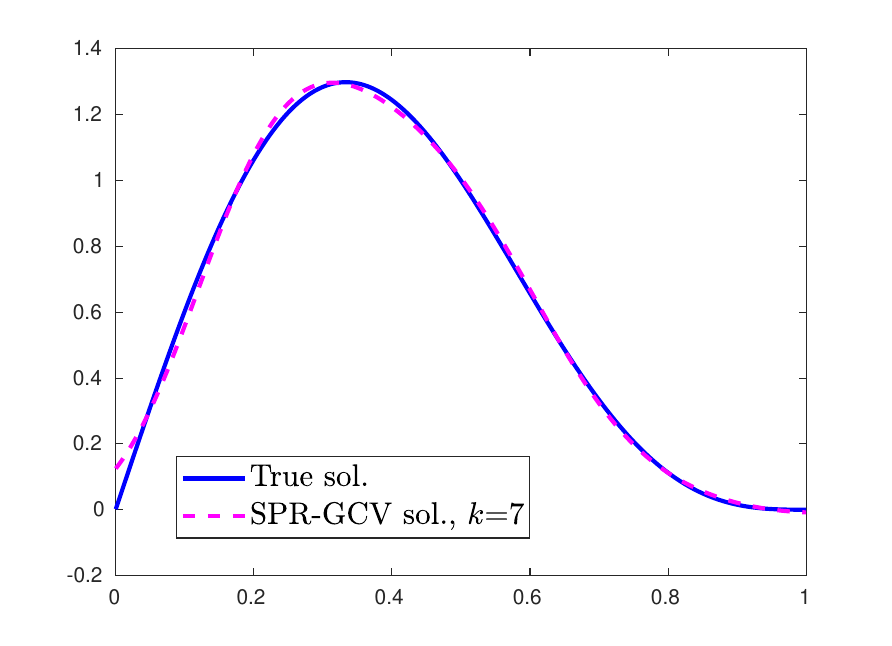}}
	\subfloat{\label{fig:3e}\includegraphics[width=0.36\textwidth]{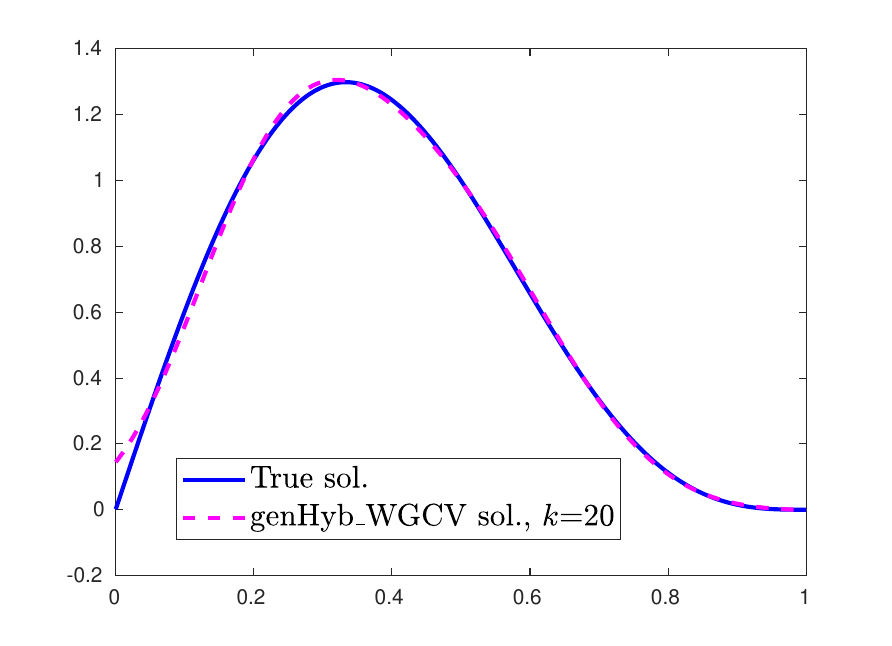}}
	\caption{Reconstructed solutions by genGKB\_SPR and genHyb\_WGCV for {\sf gravity}.}
	\label{fig3}
\end{figure}

The convergence behaviors of genGKB\_SPR and genHyb\_WGCV are shown in Figure \ref{fig2} using the relative error curves. The final regularized solutions and the corresponding iteration number are shown in Table \ref{tab1}. In this experiment, all of the DP, LC and GCV methods under-estimate the semi-convergence point $k_0$, but the estimated early stopping iterations are only slightly smaller than $k_0$, and the reconstructed solutions approximate well to the true solution. 

\begin{table}[!htbp]
	\centering
	\caption{Relative error of the final regularized solution and corresponding early stopping iteration number (in parentheses) for {\sf gravity} and {\sf shaw}.}
	\scalebox{1}{
		\begin{tabular}{llllll}
			\toprule
			Problem     & SPR-best  & SPR-DP & SPR-LC  & SPR-GCV & genHyb\_WGCV          \\
			\midrule
			{\sf gravity} & 0.0244 (8) & 0.0337 (6) & 0.0272 (7) & 0.0272 (7) & 0.0289 (20)  \\
			{\sf shaw} & 0.0487 (7) & 0.0613 (6) & 0.0983 (5) & 0.1706 (8) & 0.0761 (20)  \\
			\bottomrule
	\end{tabular}}
	\label{tab1}
\end{table}

For genHyb\_WGCV, the iterative solution tends to approximate the best Tikhonov regularization solution, but it takes much more iterations to reach a similar accuracy as the genGKB\_SPR solution at $k_0$. From the right subfigure of Figure \ref{fig2}, we find that WGCV always determines a larger regularization parameter than $\lambda_{opt}$ during the iteration, thereby computing an over-smoothed solution. Although there are currently no rigorous theoretical results, we can observe numerically that the convergence rate of the regularization parameter in genHyb\_WGCV is fast at first and then gradually slows down; this may be why the genHyb\_WGCV solution converges slowly to the best Tikhonov regularization solution. To compare the two methods, we can see from Table \ref{tab1} that for genGKB\_SPR with DP, the solution at $k=6$ is only slightly less accurate than the genHyb\_WGCV solution at $k=20$, while for genGKB\_SPR with LC or GCV the solution at $k=7$ has a slightly higher accuracy than the genHyb\_WGCV solution at $k=20$.

\paragraph{Experiment 2.}
The second small-scale problem is {\sf shaw} that models a one-dimensional image restoration using the Fredholm integral equation \eqref{fred}, where the kernel $K$ and solution $f$ are given by
\begin{align*}
	K(s, t) &= (\cos s + \cos t)^2 \left(\frac{\sin u}{u}\right)^2, \ \ \
	u = \pi(\sin s + \sin t) ,  \\
	f(t) &= 2\exp\left(-6(t-0.8)^2\right) + \exp\left(-2(t+0.5)^2\right) .
\end{align*}
The discretization approach is the same as {\sf gravity} and $m=n=2000$. We constructed a non-white Gaussian noise $\bepsilon\sim\mathcal{N}(\boldsymbol{0}, \bM)$ as follows such that $\bM$ is a diagonal matrix. First formulate a vector $\bd=(d_1,\dots,d_m)$ with elements of uniform random integers between $[1, 5]$. Then let $\bM=\mathrm{diag}(\gamma\bd)$, where $\gamma$ satisfies that the noise level $\mathbb{E}[\|\bepsilon\|_{2}]/\|\bb_{\text{true}}\|_{2}\approx \left(\mathbb{E}[\|\bepsilon\|_{2}^2]/\|\bb_{\text{true}}\|_{2}^2\right)^{1/2}=(\gamma\sum_{i=1}^{m}d_i)^{1/2}/\|\bb_{\text{true}}\|_{2}=10^{-2}$. The true solution and noisy observed data are shown in Figure \ref{fig4}.  The covariance matrix $\bN$ of the Gaussian prior of $\bx$ comes from the exponential kernel
\begin{equation*}
	\kappa_{exp}(r):= \exp\left(-(r/l)^\gamma\right),
\end{equation*}
and the parameters $l$ and $\nu$ are set as $l=0.1$ and $\nu=1$.

\begin{figure}[htbp]
	\centering
	\subfloat 
	{\label{fig:4a}\includegraphics[width=0.35\textwidth]{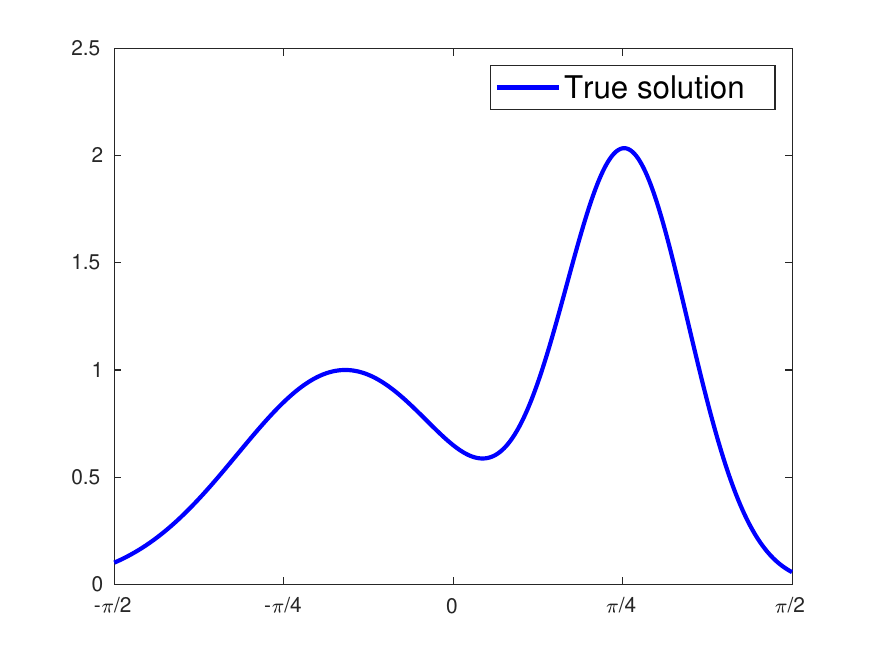}}
	\subfloat{\label{fig:4b}\includegraphics[width=0.35\textwidth]{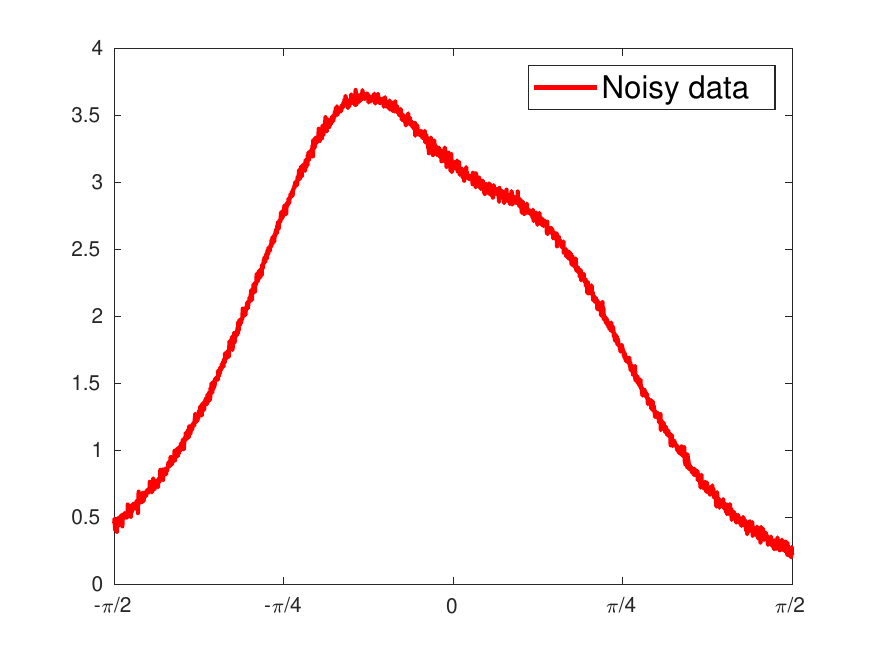}}
	\caption{Illustration of the true solution and noisy observed data for {\sf shaw}.}
	\label{fig4}
\end{figure}

\begin{figure}[htbp]
	\centering
	\subfloat
	{\label{fig:5a}\includegraphics[width=0.42\textwidth]{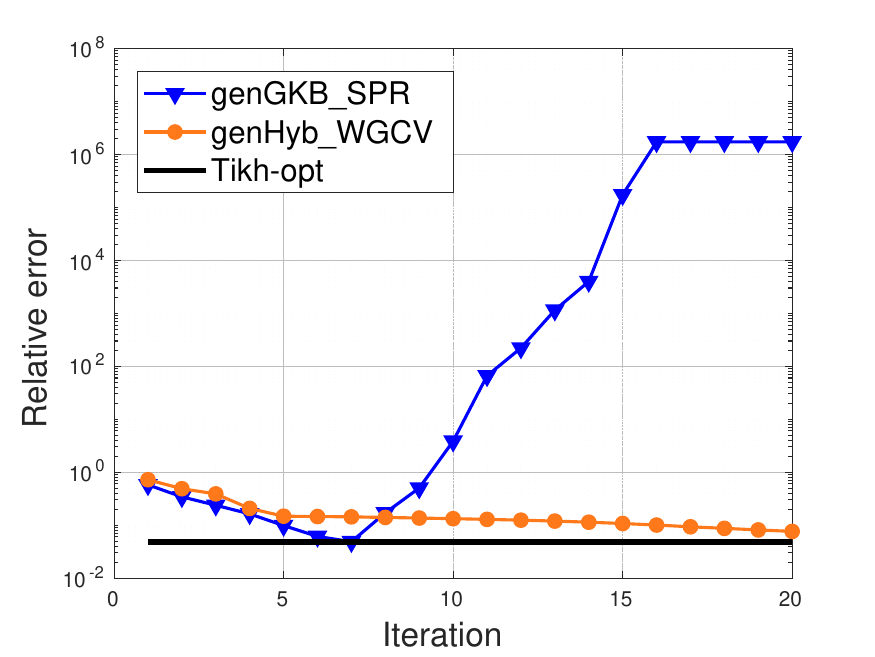}}
	\subfloat
	{\label{fig:5b}\includegraphics[width=0.42\textwidth]{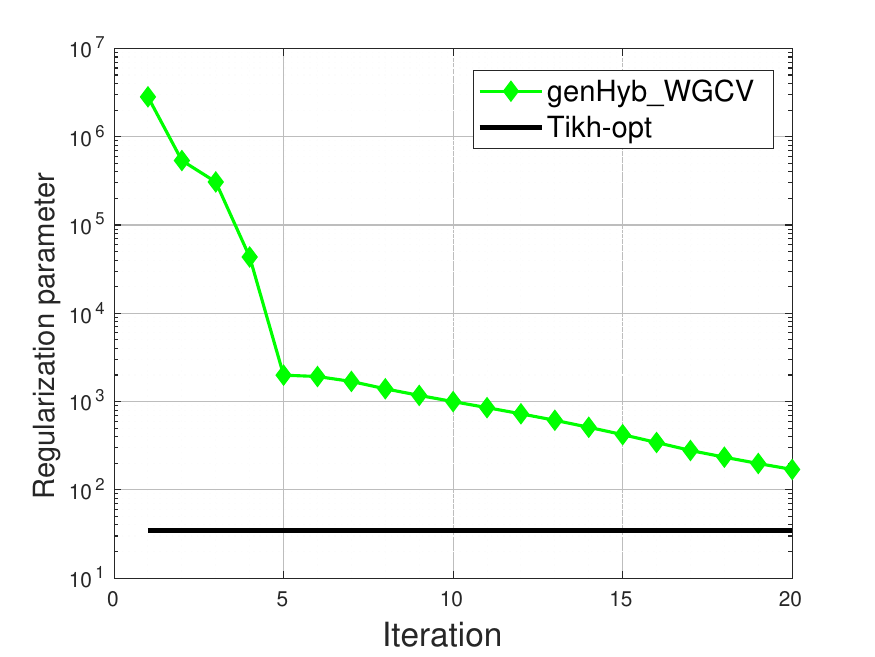}}
	\caption{Comparison of convergence behaviors between genGKB\_SPR and genHyb\_WGCV for {\sf shaw}. Left: relative error curves of regularized solution with respect to iteration number; Right: convergence trend of regularization parameters of genHyb\_WGCV towards  $\lambda_{opt}$.}
	\label{fig5}
\end{figure}

\begin{figure}[htbp]
	\centering
	\subfloat 
	{\label{fig:6a}\includegraphics[width=0.35\textwidth]{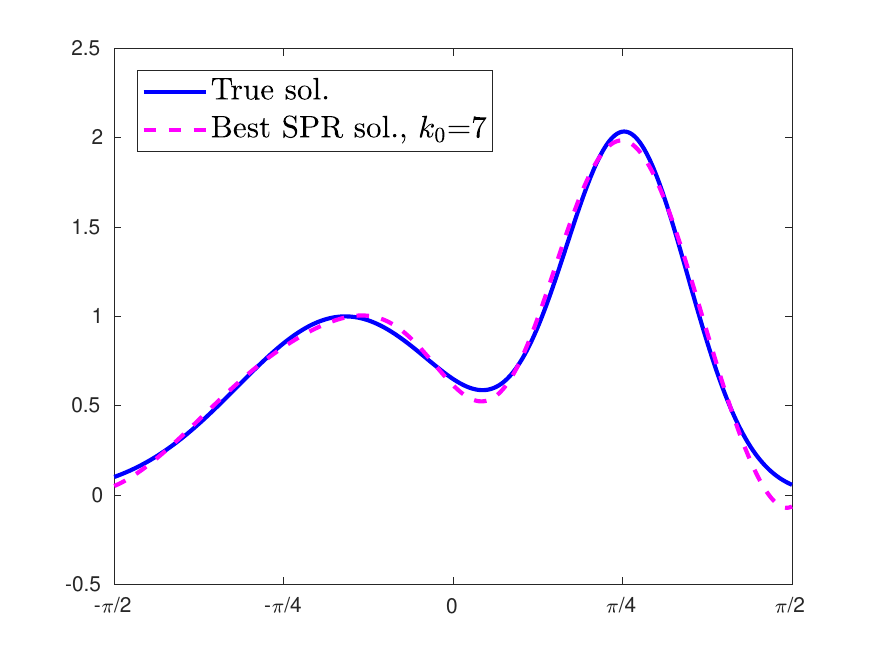}}
	\subfloat{\label{fig:6b}\includegraphics[width=0.35\textwidth]{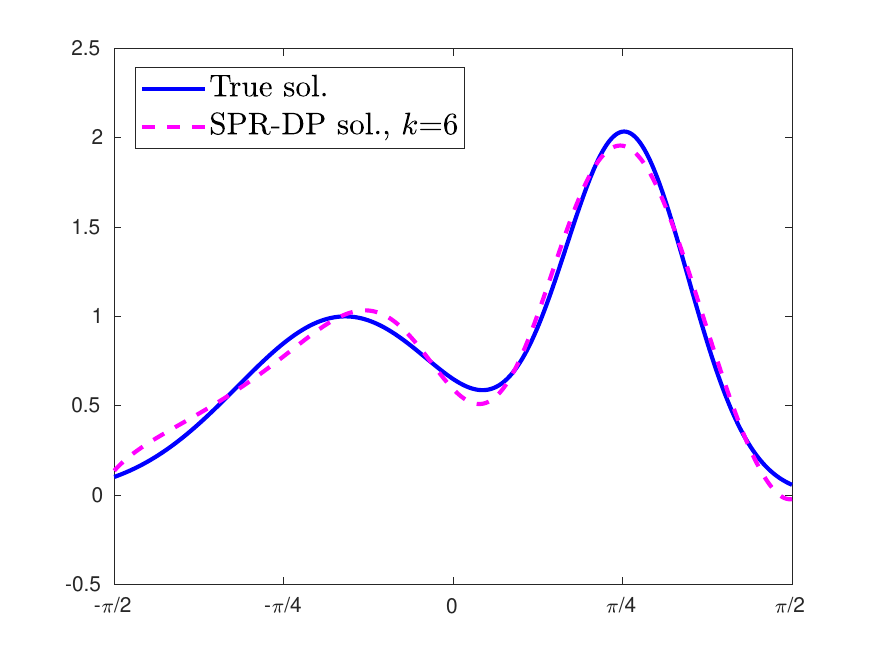}}
	\vspace{-5mm}
	\subfloat 
	{\label{fig:6c}\includegraphics[width=0.35\textwidth]{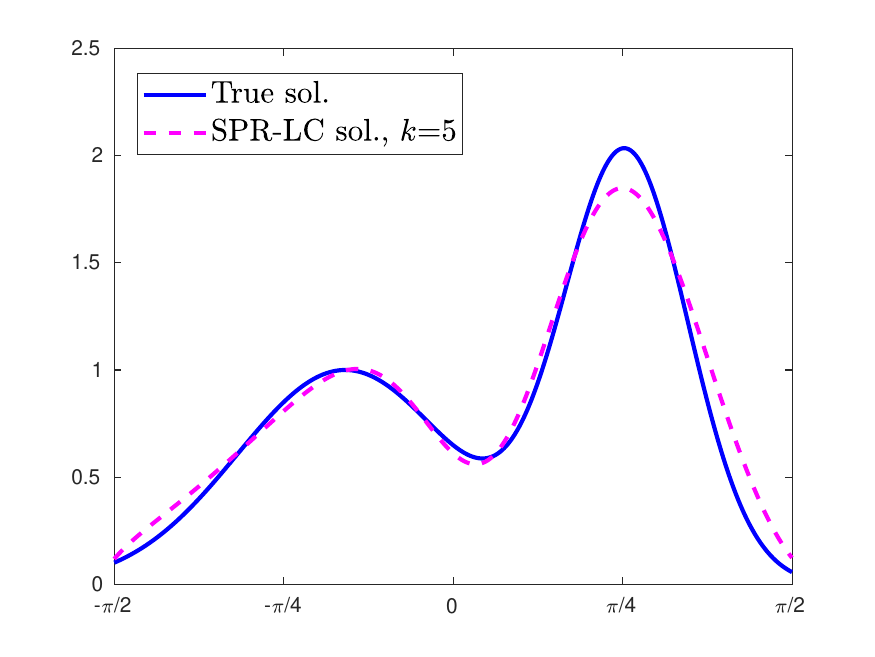}}
	\subfloat{\label{fig:6d}\includegraphics[width=0.35\textwidth]{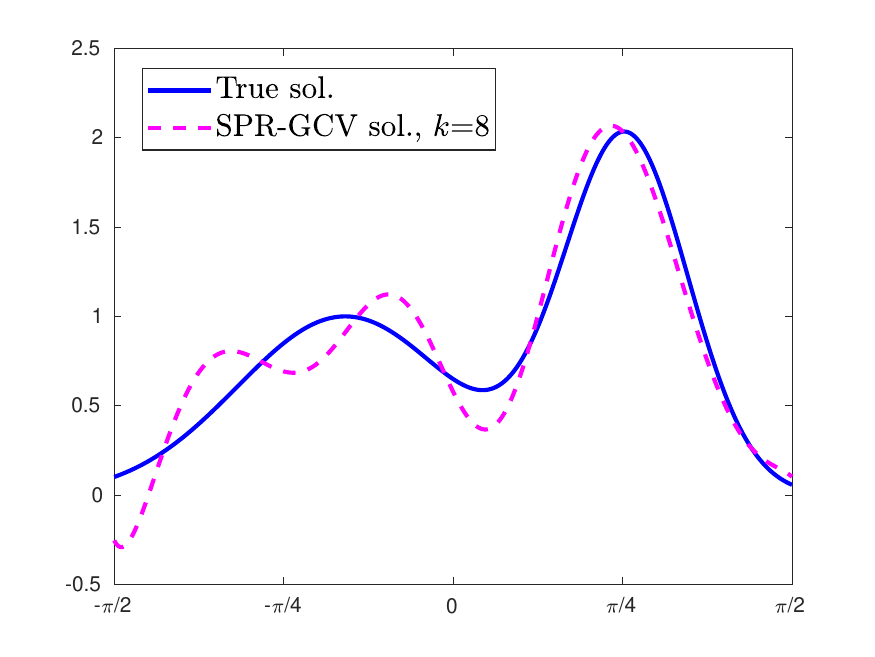}}
	\subfloat{\label{fig:6e}\includegraphics[width=0.36\textwidth]{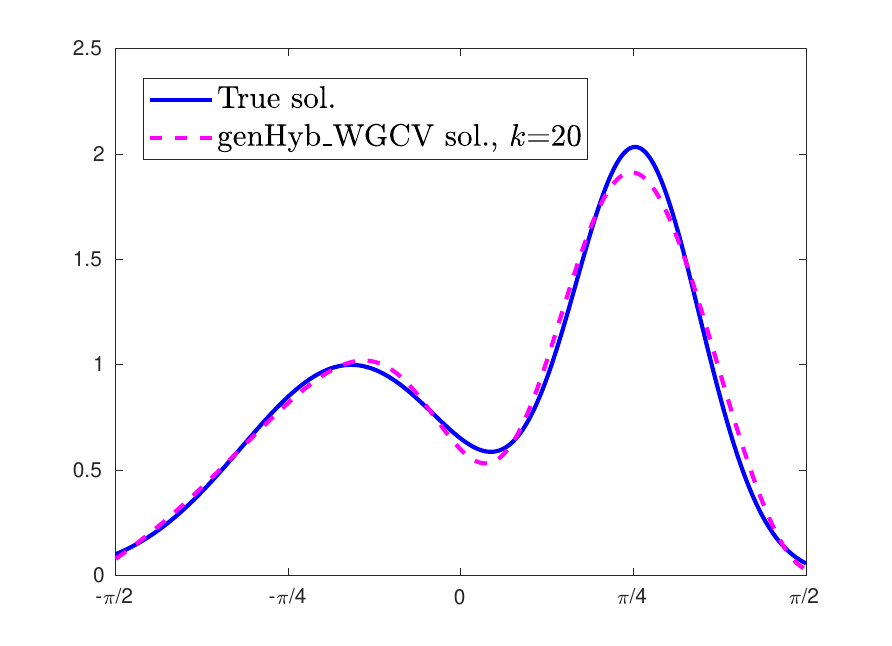}}
	\caption{Reconstructed solutions by genGKB\_SPR and genHyb\_WGCV for {\sf shaw}.}
	\label{fig6}
\end{figure}

The convergence behaviors of genGKB\_SPR and genHyb\_WGCV for {\sf shaw} are shown in Figure \ref{fig5}, while the reconstructed solutions are drawn in Figure \ref{fig6}. We can observe similar phenomena as the first experiment: the genGKB\_SPR solution exhibits the semi-convergence property, while the genHyb\_WGCV gradually converges to the best Tikhonov regularization solution and it takes much more iterations to reach a similar accuracy as the genGKB\_SPR solution at $k_0$. From Table \ref{tab1} we find that DP and LC slightly under-estimate $k_0$ while GCV over-estimate $k_0$. In this experiment, the relative error for genGKB\_SPR increases rapidly after the semi-convergence point, resulting in the GCV solution at $k=8$ with a relatively poor accuracy. However, the DP or LC solution of genGKB\_SPR at $k=6$ or $k=5$ has a satisfied accuracy, which is comparable to the genHyb\_WGCV solution at $k=20$.

\subsection{Two-dimensional inverse diffusion}
In this experiment, we consider the two-dimensional (2D) inverse diffusion problem {\sf PRdiffusion} \cite{Gazzola2019}. The forward problem is a 2D
diffusion equation in the spatial domain $\Omega=[0,1]^2$ with homogenous Neumann boundary condition and an initial function $u_0$:
\begin{equation*}
	\left\{
	\begin{aligned}
		& \frac{\partial u}{\partial t} = \Delta u, \ \ (x,y)\in \Omega, \ t>0,  \\
		& \frac{\partial u}{\partial \mathbf{n}} = 0, \ \ (x,y)\in\partial\Omega, \ t>0, \\
		& u(0,x,y) = u_0(x,y), \ \ (x,y)\in\Omega ,
	\end{aligned}
	\right.
	\qquad	
\end{equation*}
where $\mathbf{n}$ denotes the exterior normal to the boundary $\partial\Omega$. The inverse problem is to reconstruct $u_0$ from the diffusion solution $u_T=u(T,x,y)$ at time $t=T$. In the experiment, we set $T=0.005$ and $u_0$ be the following smooth function:
\begin{equation*}
	u_0(x,y) = 0.7\exp\left(-\left(\frac{x-0.4}{0.12}\right)^2-\left(\frac{y-0.5}{0.15}\right)^2\right) +
	\exp\left(-\left(\frac{x-0.7}{0.1}\right)^2-\left(\frac{y-0.8}{0.08}\right)^2\right) .
\end{equation*}
The above forward PDE is discretized on a uniform finite element mesh with $2(n_1-1)^2$ triangular elements, thereby $u$ is represented by the $n=n_{1}^2$ values at the corners of the elements, and $[0, T]$ is divided into 100 uniform time steps. The forward computation is the numerical solution of the above PDE implemented by the Crank-Nicolson-Galerkin finite element method, thereby $\bA$ is a functional handle. By choosing $n_1=128$, we have $\bA\in\mathbb{R}^{128^2\times128^2}$, which maps $\bx_{\text{true}}$ (the discretized $u_0$) to $\bb_{\text{true}}$ (the discretized $u_T$), and we add a Gaussian white noise $\bepsilon$ with noise level $10^{-3}$ to $\bb_{\text{true}}$ to get $\bb$. The true solution and noisy observed data are shown in Figure \ref{fig7}.

\begin{figure}[htbp]
	\centering
	\subfloat 
	{\label{fig:7a}\includegraphics[width=0.40\textwidth]{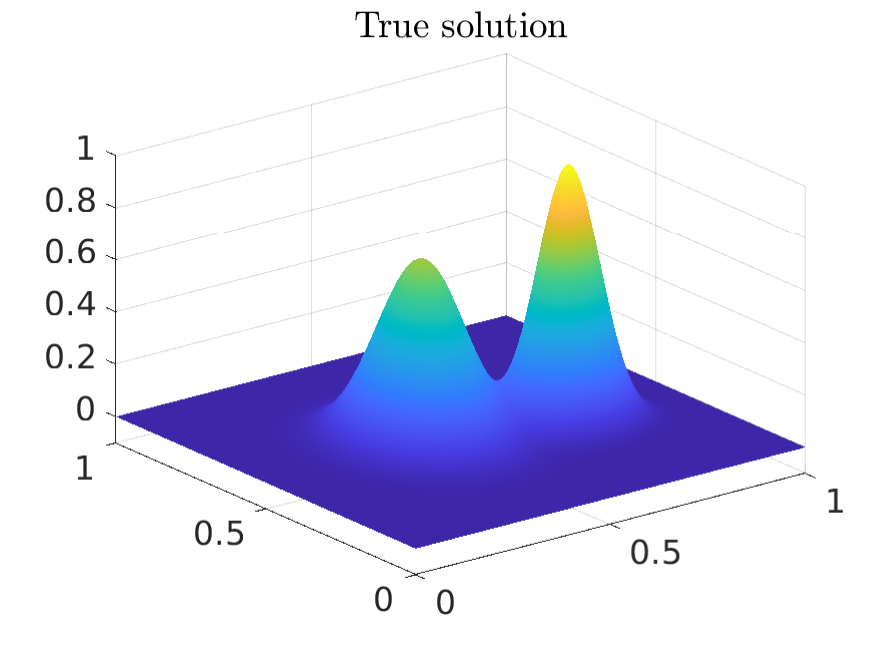}}
	\subfloat
	{\label{fig:7b}\includegraphics[width=0.40\textwidth]{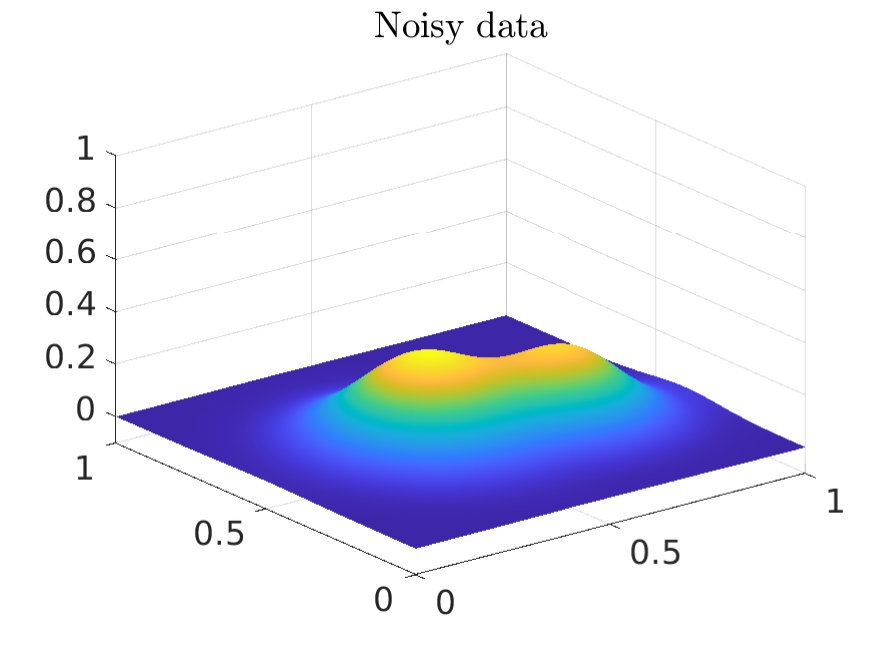}}
	\caption{Illustration of the true solution and noisy observed data for {\sf PRdiffusion}.}
	\label{fig7}
\end{figure}

In order to reconstruct the smooth initial function $u_0$, we assume a Gaussian prior of $\bx$ with the covariance matrix $\bN$ coming from the Mat\'{e}rn kernel
\begin{equation*}
	\kappa_{M}(r):= \frac{2^{1-\nu}}{\Gamma(\nu)}\Bigg(\frac{\sqrt{2\nu}r}{l}\Bigg)^\nu K_\nu\Bigg(\frac{\sqrt{2\nu}r}{l}\Bigg),
\end{equation*}
where $\Gamma$ is the gamma function, $K_{\nu}$ is the modified Bessel function of the second kind, and $l$ and $\nu$ are two positive parameters of the covariance. The parameter $\nu$ controls the smoothness of the resulting function; see \cite[\S 4.2]{Williams2006gaussian}. The smaller the value of $\nu$, the less smooth the approximated function becomes. We set $l=0.05$ and $\nu=5/2$ (for twice differentiable functions), respectively. 

\begin{figure}[htbp]
	\centering
	\subfloat{\label{fig:8a}\includegraphics[width=0.36\textwidth]{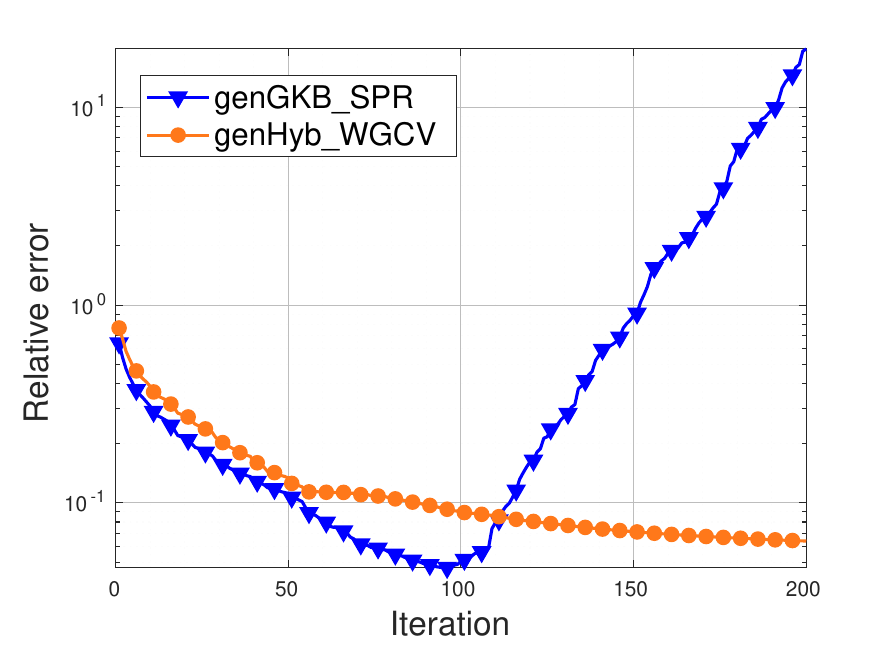}}\hspace{-5mm}
	\subfloat{\label{fig:8b}\includegraphics[width=0.35\textwidth]{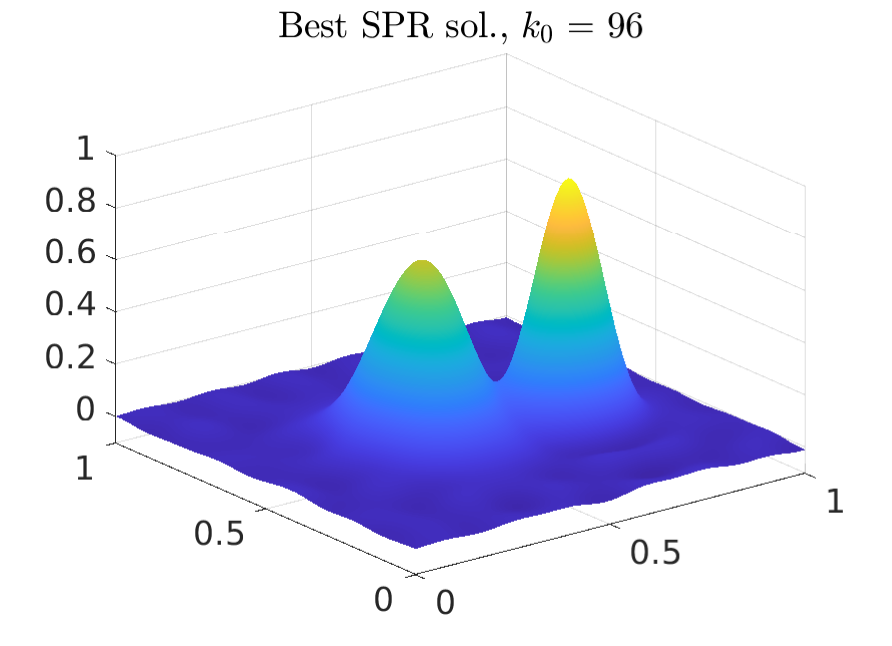}}\hspace{-6.5mm}
	\subfloat{\label{fig:8c}\includegraphics[width=0.35\textwidth]{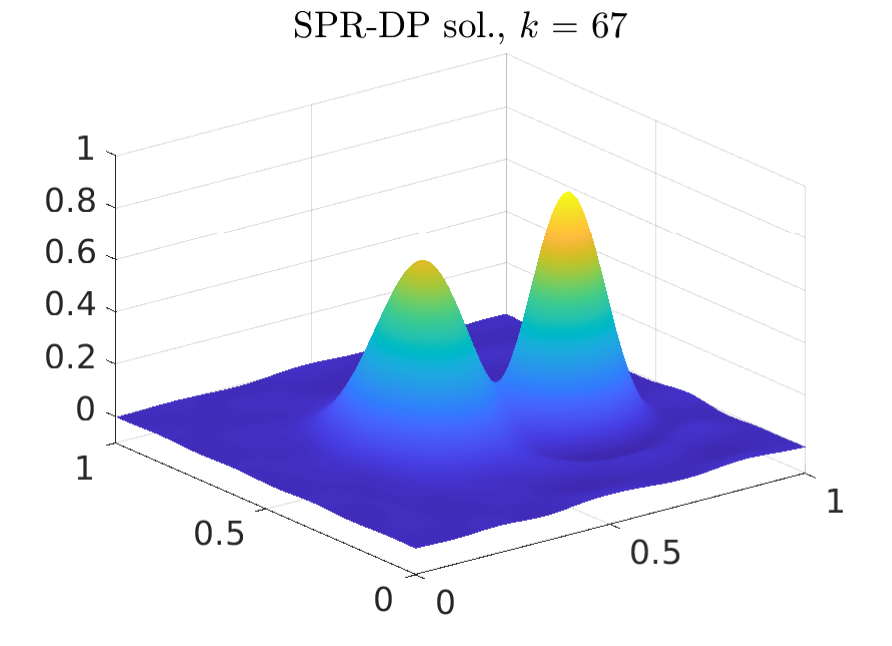}}
	\vspace{-1mm}
	\subfloat{\label{fig:8d}\includegraphics[width=0.35\textwidth]{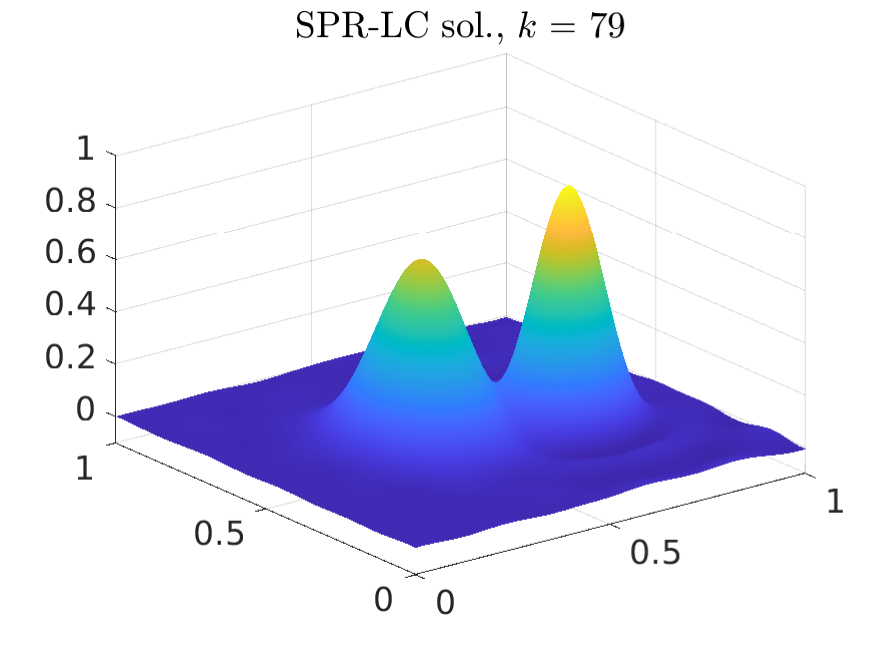}}\hspace{-6mm}
	\subfloat{\label{fig:8e}\includegraphics[width=0.35\textwidth]{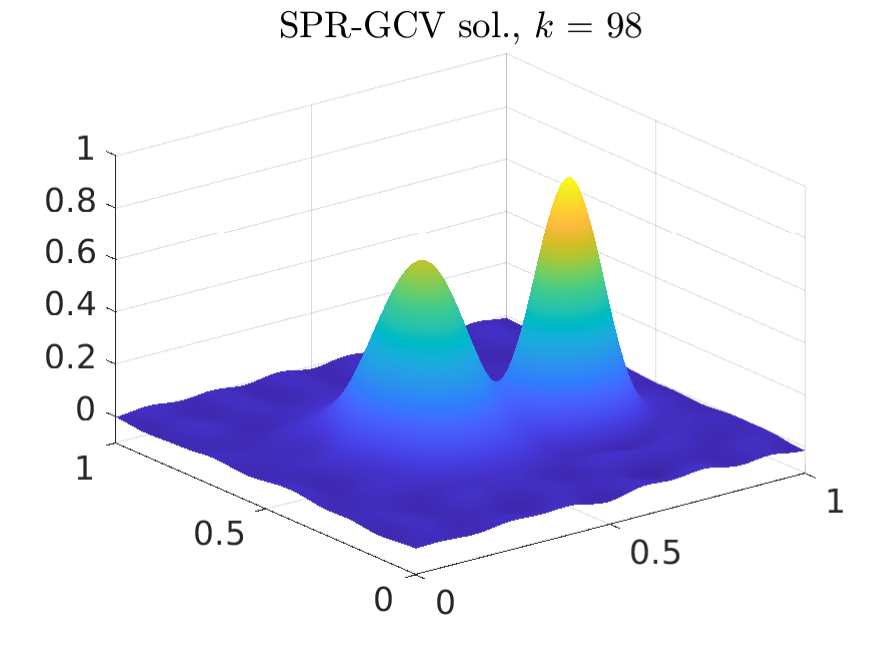}}\hspace{-6mm}
	\subfloat{\label{fig:8f}\includegraphics[width=0.35\textwidth]{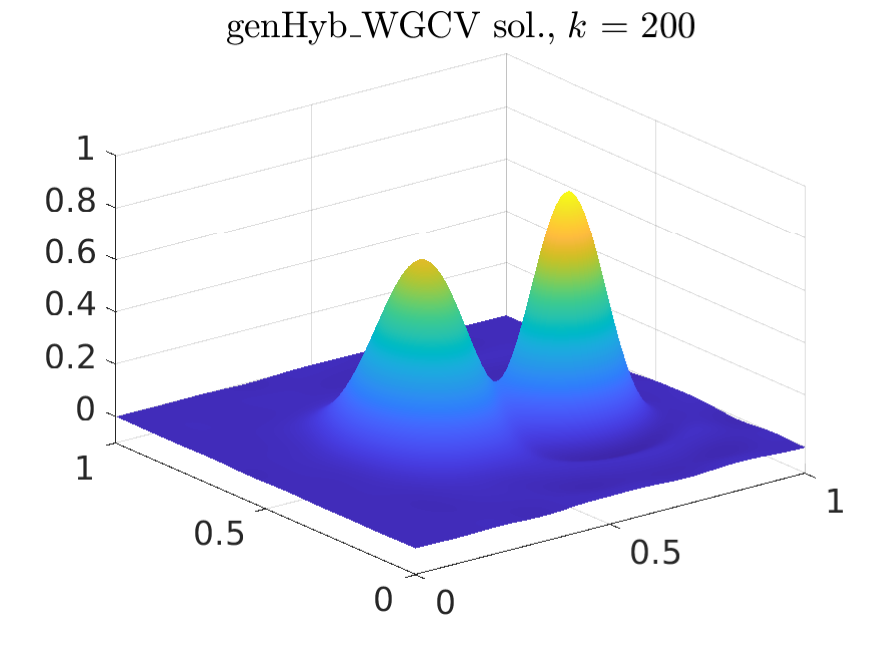}}
	\caption{Convergence behaviors and reconstructed solutions by genGKB\_SPR and genHyb\_WGCV for {\sf PRdiffusion}.}
	\label{fig8}
\end{figure}

From Figure \ref{fig8} we find genGKB\_SPR exhibits the semi-convergence behavior and the best regularized solution is obtained at $k_0=96$. The estimates of $k_0$ by DP, LC and GCV and the relative errors of corresponding solutions are shown in Table \ref{tab2}. In comparison, the genHyb\_WGCV solution convergences very slowly; its solution only reaches an accuracy slightly poorer than the best genGKB\_SPR solution even for $k=200$, much larger than $k_0$. As shown in Table \ref{tab2}, the DP, LC, and GCV solutions of genGKB\_SPR also have higher accuracy than the genGKB\_SPR solution at $k=200$. Considering this, for large-scale problems, it will take much more time for genHyb\_WGCV than genGKB\_SPR to obtain a satisfied regularized solution. We find that the reconstructed solutions by genGKB\_SPR have good smoothness and a higher similarity to $u_0$. This is due to the elaborately constructed solution subspaces by gen-GKB that incorporate prior information of $\bx$ encoded by the Mat\'{e}rn kernel to enforce the second order differentiability of a desired $\bx$.

\begin{table}[htbp]
	\centering
	\caption{Relative error of the final regularized solution and corresponding early stopping iteration number (in parentheses) for {\sf PRdiffusion} and {\sf PRspherical}.}
	\scalebox{1}{
		\begin{tabular}{llllll}
			\toprule
			Problem     & SPR-best  & SPR-DP & SPR-LC  & SPR-GCV & genHyb\_WGCV          \\
			\midrule
			{\sf PRdiffusion} & 0.0470 (96) & 0.0704 (67) & 0.0566 (79) & 0.0483 (98) & 0.0641 (200)  \\
			{\sf shaw} & 0.0026 (25) & 0.0046 (17) & 0.0033 (20) & 0.0156 (67) & 0.0106 (100)  \\
			\bottomrule
	\end{tabular}}
	\label{tab2}
\end{table}

\subsection{Two-dimensional spherical Radon transform tomography}
In this experiment, we test our method by a 2D spherical Radon transform (SRT) tomography. The SRT integrates a function over a set of spheres (or circles in the 2D case) to get a noisy observation, which plays an important role in some newly developing types of tomography such as photo-acoustic imaging \cite{klukowska2013snark09}. Here we consider the 2D SRT problem where the forward operator is generated by {\sf PRspherical} with default settings \cite{Gazzola2019}; for detailed descriptions of {\sf PRspherical} see \cite[\S 3.2]{hansen2018air}. The true solution is a $256\times 256$ pixel 2D smooth image consisting of a superposition of four Gaussian functions. Following the same approach as the experiment for {\sf shaw}, we add a Gaussian non-white noise $\bepsilon$ with noise level $10^{-2}$ to $\bb_{\text{true}}$ to get $\bb$. In the experiment $\bA\in\mathbb{R}^{92672 \times 65536}$. The true image and noisy observed data are shown in Figure \ref{fig7}. In order to reconstruct a smooth image, we assume a Gaussian prior of $\bx$ with Mat\'{e}rn kernel, where the two parameters are set as $l=500$ and $\nu=1/2$, respectively.

\begin{figure}[htbp]
	\centering
	\subfloat 
	{\label{fig:9a}\includegraphics[width=0.30\textwidth]{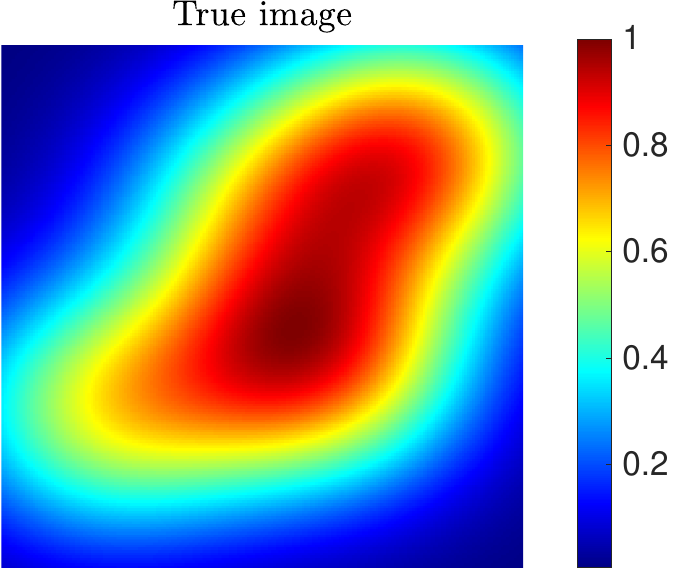}} \hspace{5mm}
	\subfloat
	{\label{fig:9b}\includegraphics[width=0.38\textwidth]{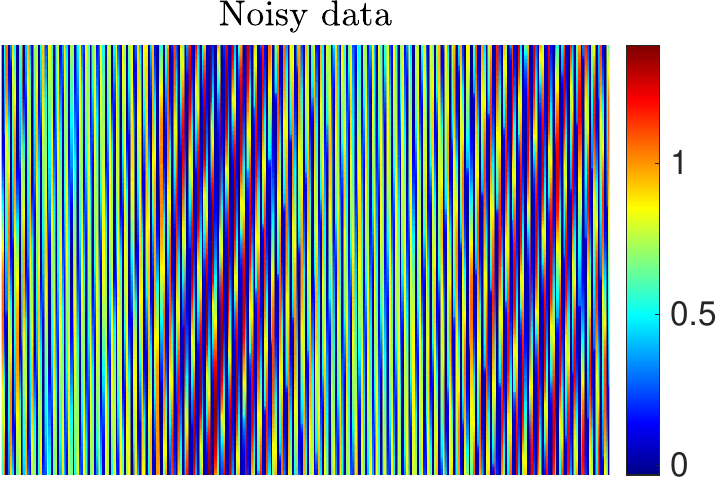}}
	\caption{Illustration of the true image and noisy observed data for {\sf PRspherical}.}
	\label{fig9}
\end{figure}

\begin{figure}[htbp]
	\centering
	\hspace{-8mm}
	\subfloat{\label{fig:10a}\includegraphics[width=0.35\textwidth]{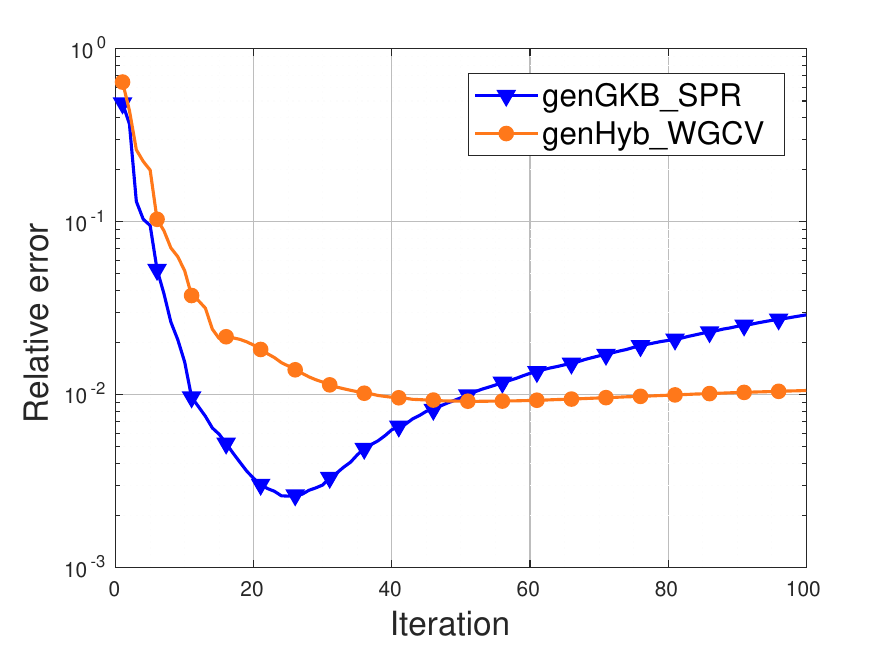}}\hspace{1mm}
	\subfloat{\label{fig:10b}\includegraphics[width=0.30\textwidth]{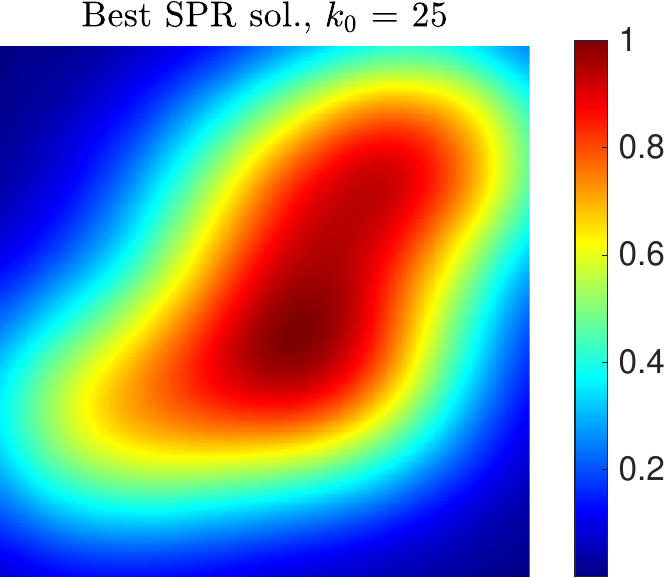}} \hspace{2mm}
	\subfloat{\label{fig:10c}\includegraphics[width=0.30\textwidth]{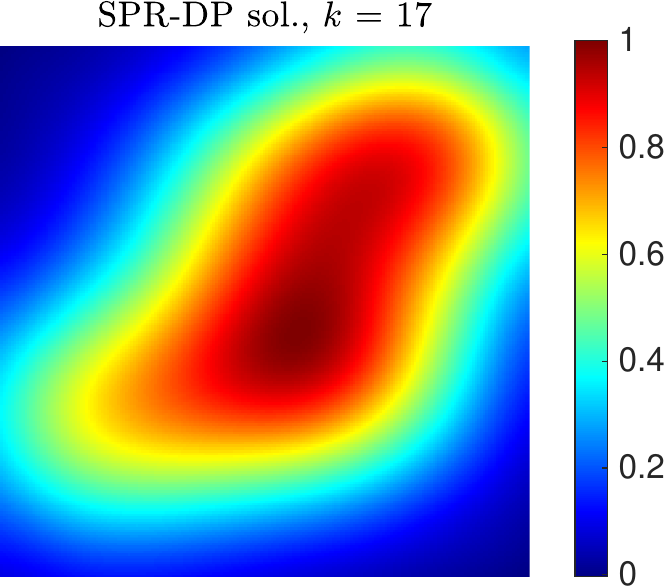}} 
	\vspace{-0.5mm} 
	\subfloat{\label{fig:10d}\includegraphics[width=0.30\textwidth]{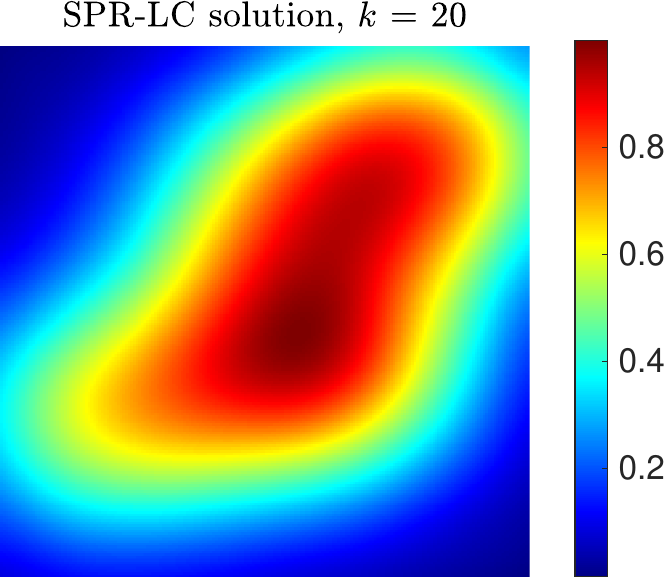}}\hspace{2mm}
	\subfloat{\label{fig:10e}\includegraphics[width=0.30\textwidth]{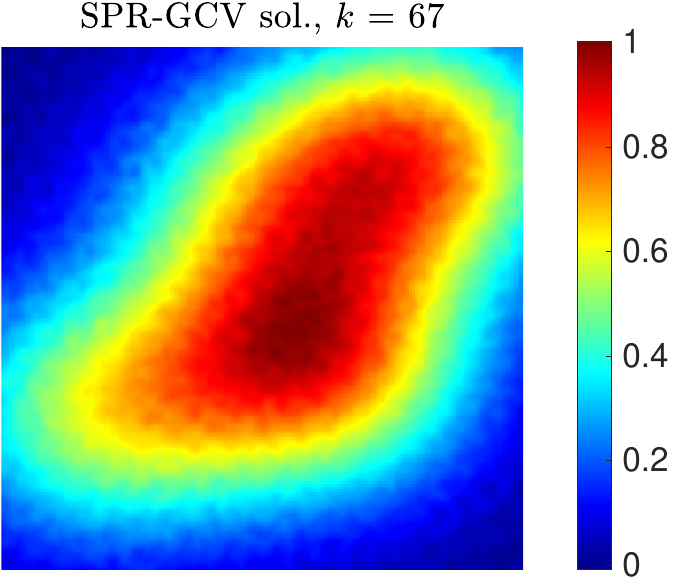}}\hspace{2mm}
	\subfloat{\label{fig:10f}\includegraphics[width=0.30\textwidth]{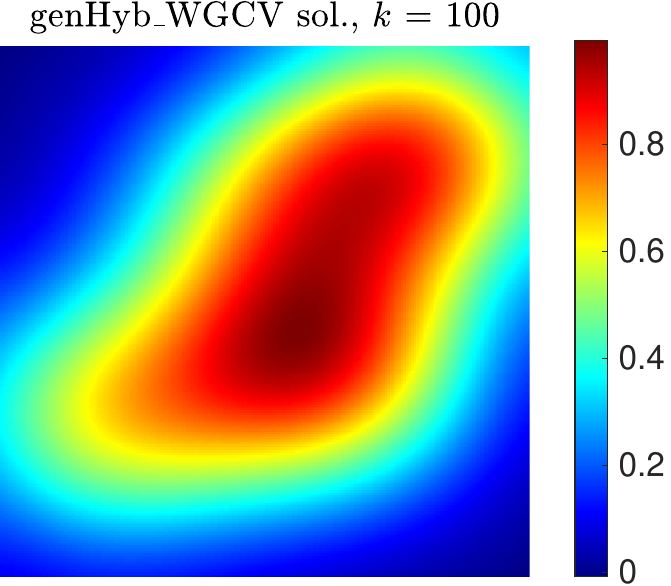}}
	\caption{Convergence behaviors and reconstructed solutions by genGKB\_SPR and genHyb\_WGCV for {\sf PRspherical}.}
	\label{fig10}
\end{figure}

The convergence behaviors of the two methods and corresponding reconstructed solutions are shown in Figure \ref{fig9}. For this large-scale inverse problems, we can find that the genGKB\_SPR method only takes 25 steps to reach the semi-convergence point, while even after 100 steps, genHyb\_WGCV only obtains a solution with significantly lower accuracy than the best genGKB\_SPR solution. If we carefully observe the convergence curve, we can notice that the relative error of the genHyb\_WGCV solution slightly increases as $k$ grows from 60 to 100. This potential instability of convergence for the hybrid methods has already been pointed out in \cite{Chungnagy2008} and \cite{Renaut2017}, which is caused by the instability of determining regularized parameter for the projected problem \eqref{hyb_k} regardless of the parameter choosing method used. Although the estimates of $k_0$ for genGKB\_SPR by DP and LC are slightly smaller than the true value, resulting in two over-smoothed solutions, we can still observe that the reconstructed images have satisfactory quality. The relative errors of them are both smaller than the genHyb\_WGCV solution at $k=100$ (also smaller than the best genHyb\_WGCV solution at around $k=60$). However, in this experiment, the GCV method suffers severely from overestimating the true $k_{0}$, and the corresponding solution deviates significantly from the best genGKB\_SPR, resulting in a very poor accuracy.

Overall, the aforementioned experiments have confirmed that the genGKB\_SPR algorithm combined with a proper early stopping rule provides a viable alternative for directly solving the Bayesian linear inverse problems \eqref{Bayes1}. The numerical tests have shown that genGKB\_SPR with an early stopping rule is more computationally effective and robust for reconstructing solutions compared to the hybrid method genHyb\_WGCV, especially for large-scale problems.

\section{Conclusion}\label{sec6}
In this paper, we have proposed a new iterative regularization algorithm for solving Bayesian linear inverse problems. Within the subspace projection regularization framework, we first introduced an iterative process that can generate a series of valid solution subspaces by treating the forward model matrix as a linear operator between the two Hilbert spaces $(\mathbb{R}^{n}, \langle\cdot,\cdot\rangle_{\bN^{-1}})$ and $(\mathbb{R}^{m}, \langle\cdot,\cdot\rangle_{\bM^{-1}})$. The original problem is projected onto these subspaces becoming a series of low dimensional linear least squares problems. We have developed an efficient procedure to update the solutions of these projected problems as approximate iterative regularization solutions, as well as update the residual norm and solution norm of them. Early stopping rules based on DP, LC and GCV have been designed, with which the iterative algorithm can get a regularized solution to the original problem with a satisfied accuracy. Several theoretical results have been established, revealing the good regularization properties of the algorithm. Numerical experiments using both small-scale and large-scale inverse problems have been conducted to demonstrate the robustness and efficiency of the algorithm. Since the most computationally intensive operations in the proposed algorithm only involve matrix-vector products, it is highly efficient for large-scale Bayesian inverse problems.



%
\section*{Declarations}
\textbf{Funding.} No funding is related to this research. \\
~\\
\textbf{Financial interests.} The authors declare they have no financial interests. \\
~\\
\noindent \textbf{Data and code availability.} The source code and data used in this work are available at \url{https://github.com/Machealb/InverProb_IterSolver}.  

%

\bibliographystyle{spmpsci}      
\bibliography{refs}

\begin{thebibliography}{10}
\providecommand{\url}[1]{{#1}}
\providecommand{\urlprefix}{URL }
\expandafter\ifx\csname urlstyle\endcsname\relax
  \providecommand{\doi}[1]{DOI~\discretionary{}{}{}#1}\else
  \providecommand{\doi}{DOI~\discretionary{}{}{}\begingroup
  \urlstyle{rm}\Url}\fi

\bibitem{antil2023efficient}
Antil, H., Saibaba, A.K.: Efficient algorithms for {B}ayesian inverse problems
  with {W}hittle--{M}at{\'e}rn priors.
\newblock SIAM J. Sci. Comput. pp. S176--S198 (2023)

\bibitem{Arioli2013}
Arioli, M.: Generalized golub--kahan bidiagonalization and stopping criteria.
\newblock SIAM J. Matrix Anal. Appl. \textbf{34}(2), 571--592 (2013).
\newblock \doi{10.1137/120866543}

\bibitem{Bai2000}
Bai, Z., Demmel, J., Dongarra, J., Ruhe, A., van~der Vorst, H.: Templates for
  the solution of algebraic eigenvalue problems: a practical guide.
\newblock SIAM (2000)

\bibitem{bjorck2015numerical}
Bj{\"o}rck, {\AA}.: Numerical methods in matrix computations, vol.~59.
\newblock Springer (2015)

\bibitem{Buzug2008}
Buzug, T.: Computed Tomography.
\newblock Springer (2008).
\newblock \urlprefix\url{https://doi.org/10.1007/978-3-540-39408-2}

\bibitem{calvetti2017priorconditioned}
Calvetti, D., Pitolli, F., Prezioso, J., Somersalo, E., Vantaggi, B.:
  Priorconditioned cgls-based quasi-map estimate, statistical stopping rule,
  and ranking of priors.
\newblock SIAM J. Sci. Comput. \textbf{39}(5), S477--S500 (2017)

\bibitem{Calvetti2018}
Calvetti, D., Pitolli, F., Somersalo, E., Vantaggi, B.: Bayes meets {K}rylov:
  Statistically inspired preconditioners for {CGLS}.
\newblock SIAM Rev. \textbf{60}(2), 429--461 (2018)

\bibitem{calvetti2005priorconditioners}
Calvetti, D., Somersalo, E.: Priorconditioners for linear systems.
\newblock Inverse Probl. \textbf{21}(4), 1397 (2005)

\bibitem{caruso2019convergence}
Caruso, N.A., Novati, P.: Convergence analysis of {LSQR} for compact operator
  equations.
\newblock Linear Algebra Appl. \textbf{583}, 146--164 (2019)

\bibitem{Chungnagy2008}
Chung, J., Nagy, J.G., O'Leary, D.P.: A weighted-{GCV} method for
  {L}anczos-hybrid regularization.
\newblock Electr. Trans. Numer. Anal. \textbf{28}(29), 149--167 (2008)

\bibitem{Chung2017}
Chung, J., Saibaba, A.K.: Generalized hybrid iterative methods for large-scale
  {B}ayesian inverse problems.
\newblock SIAM J. Sci. Comput. \textbf{39}(5), S24--S46 (2017)

\bibitem{Engl2000}
Engl, H.W., Hanke, M., Neubauer, A.: Regularization of {I}nverse {P}roblems.
\newblock Kluwer Academic Publishers (2000)

\bibitem{Gazzola2019}
Gazzola, S., Hansen, P.C., Nagy, J.G.: {IR} {T}ools: A {MATLAB} package of
  iterative regularization methods and large-scale test problems.
\newblock Numer. Algor. \textbf{81}(3), 773--811 (2019)

\bibitem{Genton2001classes}
Genton, M.G.: Classes of kernels for machine learning: a statistics
  perspective.
\newblock J. Mach. Learn. Res. \textbf{2}(Dec), 299--312 (2001)

\bibitem{Golub2013}
Golub, G.H., Van~Loan, C.F.: Matrix Computations, 4th edn.
\newblock The Johns Hopkins University Press, Baltimore (2013)

\bibitem{Golub1979}
Golub, G.H., Wahba, H.G.: Generalized cross-validation as a method for choosing
  a good ridge parameter.
\newblock Technometrics \textbf{21}(2), 215--223 (1979)

\bibitem{hadamard1923lectures}
Hadamard, J.: Lectures on Cauchy's problem in linear partial differential
  equations, vol.~15.
\newblock Yale university press (1923)

\bibitem{Hansen1989}
Hansen, P.C.: Regularization, {GSVD} and truncated {GSVD}.
\newblock BIT Numer. Math. \textbf{29}(3), 491--504 (1989)

\bibitem{hansen1990discrete}
Hansen, P.C.: The discrete {P}icard condition for discrete ill-posed problems.
\newblock BIT Numerical Mathematics \textbf{30}(4), 658--672 (1990)

\bibitem{Hansen1992}
Hansen, P.C.: Analysis of discrete ill-posed problems by means of the
  {L}-curve.
\newblock SIAM Rev. \textbf{34}(4), 561--580 (1992)

\bibitem{Hansen1998}
Hansen, P.C.: Rank-deficient and {D}iscrete {I}ll-{P}osed {P}roblems:
  {N}umerical {A}spects of {L}inear {I}nversion.
\newblock SIAM, Philadelphia (1998)

\bibitem{Hansen2007}
Hansen, P.C.: Regularization {T}ools version 4.0 for {M}atlab 7.3.
\newblock Numer. Algor. \textbf{46}(2), 189--194 (2007)

\bibitem{Hansen2010}
Hansen, P.C.: Discrete {I}nverse {P}roblems: {I}nsight and {A}lgorithms.
\newblock SIAM, Philadelphia (2010)

\bibitem{hansen2018air}
Hansen, P.C., J{\o}rgensen, J.S.: {AIR} {T}ools {II}: algebraic iterative
  reconstruction methods, improved implementation.
\newblock Numerical Algorithms \textbf{79}(1), 107--137 (2018)

\bibitem{Hansen2006}
Hansen, P.C., Nagy, J.G., O’Leary, D.P.: Deblurring {I}mages: {M}atrices,
  {S}pectra and {F}iltering.
\newblock SIAM, Philadelphia (2006)

\bibitem{Kaip2006}
Kaipio, J., Somersalo, E.: Statistical and {C}omputational {I}nverse
  {P}roblems.
\newblock Springer (2006)

\bibitem{Kilmer2001}
Kilmer, M.E., O'Leary, D.P.: Choosing regularization parameters in iterative
  methods for ill-posed problems.
\newblock SIAM J. Matrix Anal. Appl. \textbf{22}(4), 1204--1221 (2001)

\bibitem{klukowska2013snark09}
Klukowska, J., Davidi, R., Herman, G.T.: {SNARK09}--{A} software package for
  reconstruction of 2{D} images from 1{D} projections.
\newblock Computer methods and programs in biomedicine \textbf{110}(3),
  424--440 (2013)

\bibitem{Knyazev2002}
Knyazev, A.V., Argentati, M.E.: Principal angles between subspaces in an
  {A}-based scalar product: Algorithms and perturbation estimates.
\newblock SIAM J. Sci. Comput. \textbf{23}(6), 2008--2040 (2002).
\newblock \doi{10.1137/S1064827500377332}

\bibitem{Law2015}
Kody~Law Andrew~Stuart, K.Z.: Data Assimilation: A Mathematical Introduction.
\newblock Springer (2015).
\newblock \urlprefix\url{https://doi.org/10.1007/978-3-319-20325-6}

\bibitem{liesen2013krylov}
Liesen, J., Strakos, Z.: Krylov subspace methods: principles and analysis.
\newblock Numerical Mathematics and Scie (2013)

\bibitem{Morozov1966}
Morozov, V.A.: Regularization of incorrectly posed problems and the choice of
  regularization parameter.
\newblock USSR Computational Mathematics and Mathematical Physics
  \textbf{6}(1), 242--251 (1966)

\bibitem{Paige1982}
Paige, C.C., Saunders, M.A.: {LSQR}: An algorithm for sparse linear equations
  and sparse least squares.
\newblock ACM Trans. Math. Software \textbf{8}, 43--71 (1982)

\bibitem{Renaut2017}
Renaut, R.A., Vatankhah, S., Ardesta, V.E.: Hybrid and iteratively reweighted
  regularization by unbiased predictive risk and weighted {GCV} for projected
  systems.
\newblock SIAM J. Sci. Comput. \textbf{39}(2), B221--B243 (2017)

\bibitem{Richter2016}
Richter, M.: Inverse Problems: Basics, Theory and Applications in Geophysics.
\newblock Springer (2016).
\newblock \urlprefix\url{https://doi.org/10.1007/978-3-319-48384-9_3}

\bibitem{Roininen2011correlation}
Roininen, L., Lehtinen, M.S., Lasanen, S., Orisp{\"a}{\"a}, M., Markkanen, M.:
  Correlation priors.
\newblock Inverse Probl. Imag. \textbf{5}(1), 167--184 (2011)

\bibitem{Van1986rate}
Van~der Sluis, A., van~der Vorst, H.A.: The rate of convergence of conjugate
  gradients.
\newblock Numer. Math. \textbf{48}, 543--560 (1986)

\bibitem{Stewart2001matrix}
Stewart, G.W.: Matrix Algorithms, Volume II: Eigensystems.
\newblock SIAM (2001)

\bibitem{Stuart2010}
Stuart, A.M.: Inverse problems: a {B}ayesian perspective.
\newblock Acta Numer. \textbf{19}, 451--559 (2010)

\bibitem{Tikhonov1977}
Tikhonov, A.N., Arsenin, V.Y.: Solutions of {I}ll-{P}osed {P}roblems.
\newblock Washington, DC (1977)

\bibitem{Williams2006gaussian}
Williams, C.K., Rasmussen, C.E.: Gaussian processes for machine learning.
\newblock MIT press Cambridge, MA (2006)

\end{thebibliography}

\end{document}